\documentclass[14pt]{amsart}
\usepackage{amsfonts}
\usepackage{amsmath}
\usepackage{bm}
\usepackage{amsfonts,amsmath,amssymb,amscd,bbm,amsthm,mathrsfs,dsfont}
\usepackage{mathrsfs}
\usepackage{pb-diagram}
\usepackage{color}
\usepackage{amssymb}
\usepackage{xypic}
\usepackage[dvips]{graphicx}
\usepackage{hyperref}
\usepackage[all]{xy}

\newtheorem{Thm}{Theorem}[section]
\newtheorem{Lem}[Thm]{Lemma}

\newtheorem{Def}[Thm]{Definition}
\newtheorem{Cor}[Thm]{Corollary}
\newtheorem{Prop}[Thm]{Proposition}
\newtheorem{Ex}[Thm]{Example}
\newtheorem{Rem}[Thm]{Remark}
\newtheorem{Fact}[Thm]{Fact}
\newtheorem{Conj}[Thm]{Conjecture}
\title{On representation rings in the context of\\ monoidal categories  }

\author{Min Huang $\;\;\;\;\;\;$ Fang Li $\;\;\;\;\;\;$ Yichao Yang}

\address{Min Huang
\newline Department
of Mathematics, Zhejiang University (Yuquan Campus), Hangzhou, Zhejiang
310027,  P.R.China}
\email{minhuang2007@gmail.com}

\address{Fang Li
\newline Department
of Mathematics, Zhejiang University (Yuquan Campus), Hangzhou, Zhejiang
310027, P.R.China}
\email{fangli@zju.edu.cn}

\address{Yichao Yang
\newline Department
of Mathematics, Zhejiang University (Yuquan Campus), Hangzhou, Zhejiang
310027,  P.R.China}
\email{yyc880113@163.com}

\date{version of \today}

\newcommand{\lra}{\longrightarrow}

\newcommand{\ra}{\rightarrow}
\newcommand{\sdp}{\times\kern-.2em\vrule height1.1ex depth-.05ex}
\newcommand{\epi}{\lra \kern-.8em\ra}

 \normalbaselines
\begin{document}

\renewcommand{\thefootnote}{\alph{footnote}}
\setcounter{footnote}{-1} 
\renewcommand{\thefootnote}{\alph{footnote}}
\setcounter{footnote}{-1} 

\begin{abstract}
In general, representation rings are well-known as Green rings from module categories of Hopf algebras.

In this paper, we study Green rings in the context of
monoidal categories such that representations of Hopf algebras can be
investigated through Green rings of various levels from module
categories to derived categories in the unified view-point. Firstly,
as analogue of representation rings of Hopf algebras, we set up the
so-called Green rings of monoidal categories, and then list some such
categories including module categories, complex categories, homotopy
categories, derived categories and (derived) shift categories, etc.
and the relationship among their corresponding Green rings.

The main part of this paper is to characterize
representation rings and derived rings of a class of inite dimensional Hopf
algebras constructed from the Nakayama
truncated algebras $KZ_{n}/J^{d}$ with certain constraints. For the
representation ring $r(KZ_{n}/J^{d})$, we completely determine
its generators and the relations of generators via the method of Pascal triangle. For the
derived ring $dr(KZ_{n}/J^{2})$(i.e., $d=2$), we determine its generators and give the relations of generators. In these two aspects, the polynomial characterizations of the representation ring and the derived ring are both given.
\end{abstract}

\maketitle
\bigskip
\tableofcontents

\section{Introduction}

Throughout this paper let $K$ be an algebraically closed field, all
modules are left modules. As we all know, the Grothendieck group
$K_{0}(A)$ plays an important role for a $K-$algebra $A$. If $A$ is
not only an algebra, but also a Hopf algebra, we can define a ring
structure on $K_{0}(A)$ by using the comultiplication $\Delta$, that
is, the Grothendieck ring. However, Grothendieck ring
only reflects the relations among irreducible representations, but not
all indecomposable representations. Therefore, in some
references (e.g. \cite{[6]},\cite{[13]}), representation rings of finite groups were introduced to study modular representation theory through the
structure of representation rings such as semi-simplicity, etc..

In essence, representation rings are constructed by (indecomposable) objects from the module categories of Hopf algebras and then are used as a tool to reflect the property of Hopf algebras. Hence, we can apply the same idea to more categories enjoying similarity, in which monoidal
categories seem be the most natural one. So, in this paper, we firstly introduce the concept of Green rings over
monoidal categories which we point out to be
regarded as the decategorification of monoidal categories, or conversely
monoidal categories can be thought as the categorification of Green rings.

The view-point of Green rings of monoidal categories is important
for us, since under which we can unify a series of special monoidal categories via Green rings. Representation rings over module categories are of the most importance. The others include that for a Hopf algebra $H$, the complex category $Ch(H)$, the homotopy category $K(H)$, the derived category $D^b(H)$ and their subcategories $Ch^{sh}(H)$ and $D^{sh}(H)$ with Green rings respectively called as complex rings, derived rings and (derived) shift rings, etc. In this part, we also give the relations among the Green rings of these special monoidal categories.

Green rings mentioned above are non-commutative unless
the Hopf algebras are cocommutative. Therefore in general, we can characterize these rings via free algebras; specially, when the
Hopf algebras are cocommutative, we can use polynomial algebras to characterize the Green rings. At the end of this part, we give the
characterizations of shift rings and derived rings by using of free algebras and polynomial algebras.

 Most of the content of this paper is to characterize representation rings and derived rings of a special class of finite dimensional Hopf
algebras, i.e. the Nakayama truncated algebras $KZ_{n}/J^{d}$ with $d=p^{m}\leq n$ over an algebraically closed field $K$ of characteristic $p$.
Nakayama algebras, which are also called generalized uniserial algebras, were introduced by Nakayama who characterized this kind of algebras as an artinian ring $R$ whose each $R$-module is a direct sum of quasi-primitive modules (i.e. those factor modules of primitive one-sided ideals), see \cite{[1]},\cite{[10]},\cite{[20]}. The Nakayama truncated algebras, which were also studied in \cite{[2]},\cite{[3]},\cite{[pre3]}, are always basic self-injective Nakayama
algebras. Some well-known algebras can be realized as Nakayama truncated algebras, such as the Taft algebras and the generalized Taft algebras studied in \cite{[21]},\cite{[24]}. Another reason why we consider this kind of Nakayama truncated algebras is that these algebras possess Hopf algebra structure, see Proposition \ref{prop3.2}.

The representation rings of certain Hopf algebras are computed, such as
finite dimensional semi-simple Hopf algebra, the enveloping algebra
of a complex semi-simple Lie algebra, the polynomial Hopf algebra
$K[x]$ and the Sweedler 4-dimensional Hopf algebra (see,
e.g. \cite{[8]},\cite{[9]},\cite{[11]},\cite{[18]},\cite{[26]}).
Here we would like to mention the recent work by Chen, Oystaeyen and
Zhang for Taft algebras and by Li and Zhang for the generalized Taft
algebra (see \cite{[7]},\cite{[17]}).

Although all the algebraic structure of Taft algebras, generalized
Taft algebras and Nakayama truncated algebras can be written as the
form $KZ_{n}/J^{d}$ for an oriented cycles $Z_{n}$ in the cases of
$d=n,d|n$ and $d=p^m\leq n$ respectively, however, they still have a
lot of differences. Actually, the differences between our paper and \cite{[7]},\cite{[17]}
 (that is, the representation rings of Taft algebras and generalized
Taft algebras) can be visualized in the following table.

\setlength{\tabcolsep}{1.5pt}
\renewcommand{\arraystretch}{1.3}
\begin{center}
\begin{tabular}{|c|c|c|c|}
\hline
\multicolumn{1}{|c|}{ } & \multicolumn{1}{|c|}{Field ~ K} & \multicolumn{1}{|c|}{Relation~
between
~$n$ and ~$d$} & \multicolumn{1}{|c|}{Number ~of ~Generators} \\
\hline
$[7]$ & Arbitrary & $n=d$ & $2$ \\
\hline
$[17]$ & $\{\overline{K}=K\}$ $+$ $\{${\rm char}$K=0$$\}$ & $d|n$ & $2$  \\
\hline
${\rm Here}$ & $\{\overline{K}=K\}$ $+$ $\{${\rm char}$K=p$$\}$ & $\{n\geq d\}$ $+$
$\{d=p^{m}$
for integer $m>0$\} & $m+1$  \\
\hline
\end{tabular}
\end{center}
\bigskip

If we fix a natural integer $n$, by the Prime number theorem and the
Fundamental
theorem of arithmetic, we have \[\begin{array}{ccl}
\sharp \{d~|~d\leq n,d=p^m ~ \text{for}~ m \geq 0~ \text{and a prime}~ p \} & \geq &
 {\rm Pi}(n)\approx \frac{n}{\log{n}} \\
  & \geq & (\alpha_{1}+1)\cdots(\alpha_{s}+1) \\
  & = & \sharp
\{d~ |~ d \text{ is a divisor of}~ n\} \\
\end{array}\]
 for any natural integer $n=p_{1}^{\alpha_{1}}\cdots
p_{s}^{\alpha_{s}}$, where $\sharp $ denotes the cardinally of the
set  and ${\rm Pi}(n)$ is the number of primes less than or equal to
$n$. It means that the number of the cases covered in this paper is larger than
that of the cases considered in \cite{[7]},\cite{[17]}.

The Nakayama truncated algebra $KZ_{n}/J^{d}$ is of finite representation type and becomes a cocommutative Hopf algebra in this case we discuss. We give the generators of its representation ring with cardinality less than the number of iso-classes of indecomposable modules and determine the relations of these generators. In order to express more clearly, we introduce the method of Pascal triangle. At last, we use linear recurrent sequence to obtain the polynomial characterization of the representation ring $r(KZ_{n}/J^{d})$.

Our other attempt is to calculate the derived ring of the algebra $KZ_{n}/J^{d}$. In fact,
the representation ring $r(KZ_{n}/J^{d})$ which have been characterized completely at Section 7 is a subring of the derived
ring $dr(KZ_{n}/J^{d})$. Another subring of $dr(KZ_{n}/J^{d})$ is the shift ring $sh(KZ_{n}/J^{d})$ possessing the graded structure which is easy to be described, see Proposition \ref{prop10.3}. However, the calculation of the structure coefficients of the derived ring $dr(KZ_{n}/J^{d})$ is more difficult than that of $sh(KZ_{n}/J^{d})$ and $r(KZ_{n}/J^{d})$. Indeed, so far we are only able to investigate the derived ring for the case $d=2$, i.e. $H=KZ_{n}/J^{2}$, see Theorem \ref{thm10.6} where the generators of $dr(KZ_{n}/J^{2})$, the relations of generators and the
polynomial characterization are given respectively in (i), (ii) and (iii). The remainder problem is that
 in Theorem \ref{thm10.6}, for the case $s>s'>1$, the decomposition of the double complex
$P^{\bullet}(j'+s'-1,j')\otimes P^{\bullet}(j+s-1,j)$ has three
choices for them and we still cannot determinate actually which one is. For this problem, we conjecture that there can be only one choice of decomposition of the double complex $P^{\bullet}(j'+s'-1,j')\otimes P^{\bullet}(j+s-1,j)$, see Conjecture \ref{conj10.4}.

As an application of the polynomial characterization of representation rings, we give
a preliminary discussion of the isomorphism problem on representation
rings. For this, the complexified representation ring $R(KZ_{n}/J^{d})$ is defined from $r(KZ_{n}/J^{d})$ and then using it, a sufficient condition for such two representation rings $R(KZ_{n_{1}}/J^{d_{1}})$ and $R(KZ_{n_{2}}/J^{d_{2}})$ with $n_{1}d_{1}=n_{2}d_{2}$ to be isomorphic is given that in their polynomial characterizations, the ideals generated by the relations in the polynomial rings are both radical ideals.

This paper is organized as follows. In Section $2$, we
first define Green rings for monoidal categories and
then introduce representation rings, (derived) shift rings and derived ring as
some special examples of Green rings; moreover, their relations are discussed.
 In Section $3$, we recall some basic definitions of Nakayama truncated
algebras $KZ_{n}/J^{d}$ and list all the indecomposables
modules and then show the Hopf algebra structure of $KZ_{n}/J^{d}$ due to some results from $[1,12]$. In Section $4$, we give an
elementary lemma to simplify our proof in the process. In Section $5$, for
integers $0\leq i,i',l\leq d-1$ and a sequence of integers
$U=\{u_{j}\}_{0\leq j\leq l}$, we give a combinatorial way to
construct the indecomposable submodule
$M(i,i',l,\{u_{0},u_{1},\cdots,u_{l}\})$ of the tensor product of
indecomposable modules $M(i,\overline{0})\otimes M(i',\overline{0})$ via Pascal triangle.
In Section $6$, we describe the generators  and relations of the representation ring $r(KZ_{n}/J^{d})$ of
Nakayama truncated algebras $KZ_{n}/J^{d}$ precisely. In Section $7$,  the
representation ring $r(KZ_{n}/J^{d})$ is characterized as the certain quotient ring of the corresponding polynomial ring. In Section $8$, we
first give the structure of the shift ring $sh(KZ_{n}/J^{d})$ and then calculate the derived ring $dr(KZ_{n}/J^{2})$ through the generators and their relations, although the same discussion for general $dr(KZ_{n}/J^{2})$ still keeps deeply difficult so far.
\bigskip

\section{Analogue of Green rings from monoidal categories and some special cases }

In this section, we put the theory of representation rings (i.e. Green rings) of
module categories in the context of monoidal categories.

Monoidal categories (also called tensor categories) were introduced
in 1963 by B$\grave{e}$nabou \cite{[5]}, which are used to define the
concept of a monoid object and an associated action on the objects
of the category. They arealso used in enriched categories. The
theory of monoidal categories has numerous applications, e.g. to
define models for the multiplicative fragment of intuitionistic
linear logic, and to form the mathematical foundation for the
topological order in condensed matter. Braided monoidal categories
have applications in quantum field theory and string theory.

We will view the method of Green rings in the meaning of monoidal
categories and then in particular, introduce the analogue methods
via module categories, complex categories, homotopy categories, derived categories and the (derived) shift categories, where the last one can be viewed as the full subcategories of complex categories (respectively, bounded derived categories).

A {\em monoidal category} $(\mathcal C, \otimes, I, a, l, r)$ is a
category $\mathcal C$ which is equipped with a {\em tensor product}
$\otimes: \mathcal C\times \mathcal C\rightarrow \mathcal C$, with
an object $I$, called the {\em unit} of $\mathcal C$, with an
associativity constraint $a$, a {\em left unit constraint} $l$ and a
{\em right unit constraint} $r$ with respect to $I$ such that the
Pentagon Axiom and the Triangle Axiom are satisfied, see \cite{[16]}.

Throughout this paper, we always assume the monoidal category $\mathcal C$ to be an additive
Krull-Schmidt category and its tensor product $\otimes$ to be an additive bifunctor.

\begin{Def}\label{generalgreen}
{\rm Assume a monoidal category $(\mathcal C, \otimes, I, a, l, r)$ is an additive
Krull-Schmidt category with $\otimes$ as an additive bifunctor. Denote by $[C]$ the isomorphism class of an
object $C$ in $\mathcal C$. Then the {\em Green ring} $gr(\mathcal C)$ of
$\mathcal C$ is defined as the ring with identity $[I]$ which is an
additive abelian group generated by all $[C]$ modulo the relations
$[C]+[D]=[C\oplus D]$ with multiplication given by $[C][D]=[C\otimes
D]$ for any $C,D\in Ob(\mathcal C)$.}
\end{Def}

Indeed, the Green ring $gr(\mathcal C)$ can be viewed as the decategorification of the monoidal category $\mathcal C$ in the categorification theory as we will show below.

Recall that an {\em additive 2-category} is a category enriched
over the category of additive categories, i.e. for any two objects
$i,j$, the morphism set between $i$ and $j$ endows with an additive
category structure which satisfies some axioms. For details, refer
to \cite{[19]}.

The {\em split Grothendieck category} $[\mathscr{C}]_{\oplus}$ of a 2-category $\mathscr{C}$ is the
category which has the same objects with that of $\mathscr{C}$ and for any $i,
j \in [\mathscr{C}]$, we have ${\rm Hom}_{[\mathscr{C}]_{\oplus}}(i,
j)=[\mathscr{C}(i,j)]_{\oplus}$, where $[\mathscr{C}(i,j)]_{\oplus}$
denotes the split Grothendieck group of $\mathscr{C}(i,j)$, that is,
the quotient of the free abelian group generated by the isomorphic
classes of $\mathscr{C}(i,j)$ modulo the relations $[Y]-[X]-[Z]$
when $Y\cong X\oplus Z$, with the multiplication of morphisms is
given by $[M]\circ[N]=[M\circ N]$.

An additive 2-category
$\mathscr{C}$ is called a {\em categorification} of a $K$-linear
category $\mathscr{D}$ if $K\otimes_{\mathbb{Z}}[\mathscr{C}]$ is
isomorphic to $\mathscr{D}$. In this case, the category
$\mathscr{D}$ is called a  {\em $K$-decategorification} of
$\mathscr{C}$.

Clearly, an additive Krull-Schmidt monoidal category
$\mathcal {C}$ with  an additive bifunctor $\otimes$ can be
regarded as an additive 2-category $\mathscr{C}$ which has only one object $i$ and the morphism class $\mathscr{C}(i,
i)=\mathcal {C}$ with the multiplication of morphisms as $M\circ
N=M\otimes N$ for any $M, N \in \mathcal {C}$. We denote this 2-category as $\mathscr{C}_{\mathcal {C}}$. Moreover, a $K$-algebra $A$ can be regarded as a linear $K$-category $\mathcal {A}$ which has only one object $i$ and the morphism class $\mathcal {A}(i,i)=A$ with the multiplication of morphisms as
$a\circ b=ab$ for any $a, b \in A$.

From those viewpoints above, for a monoidal category $\mathcal C$ in
Definition \ref{generalgreen} and its corresponding 2-category
$\mathscr{C}_{\mathcal C}$, we have the ring isomorphism
$gr(\mathcal C)\cong [\mathscr{C}_{\mathcal C}]_\oplus$. After
expanding $gr(\mathcal C)$ to be over $K$, denoted by $gr_K(\mathcal
C)=K\otimes_{\mathbb Z} gr(\mathcal C)$ (called the {\em $K$-Green
ring}) as a $K$-linear category, we obtain that $gr_K(\mathcal
C)\cong K\otimes_{\mathbb Z}[\mathscr{C}_{\mathcal C}]_\oplus$. In
summary, we have:

\begin{Thm}
{\rm For a monoidal category $\mathcal C$,  its $K$-green ring $gr_K(\mathcal C)$ is a $K$-decategorification of $\mathcal C$; equivalently saying, $\mathcal C$ is a categorification of the $K$-Green ring $gr_K(\mathcal C)$.}
\end{Thm}

Some properties of the Green ring $gr(\mathcal C)$ of a monoidal category $\mathcal C$ are listed as
follows:

(i)~ $gr(\mathcal C)$ is an associative ring with identity $[I]$
for the unit object $I$ in $\mathcal C$.
$[I]$ is the identity of  $gr(\mathcal C)$ since $I$ is the unit
object of $\mathcal C$. The associativity of $gr(\mathcal C)$
follows the associativity constraint $a$ of $\mathcal C$; its distributive law is true due to that $\otimes$ is an additive bifunctor.

(ii)~ Under the addition, $gr(\mathcal C)$ is a free abelian group
with a $\mathds Z$-basis $\{[D]:~ D\in {\rm ind}\mathcal C\}$ where
${\rm ind}\mathcal C$ denotes the subcategory consisting of
indecomposable objects in $\mathcal C$.

For an arbitrary $K$-Hopf algebra $H$, all of the finitely generated
module category $_Hmod$, the complex category $Ch(H)$, the homotopy
category $K(H)$, the bounded derived category $D^b(H)$ are monoidal
categories. Meantime, we will construct the new monoidal categories, that is, two full subcategories $Ch^{sh}(H)$ and $D^{sh}(H)$ of $Ch(H)$ and $D^b(H)$ respectively, where $Ob(Ch^{sh}(H))=Ob(D^{sh}(H))=M^{\bullet}[s], s\in\mathds{Z}$, $M\in {\rm ind}(_Hmod)$ and their Green rings are the same, called as the {\em shift ring} $sh(H)$.

In this paper, we will mainly study the categories $_Hmod$,
$D^b(H)$, $Ch^{sh}(H)$ and $D^{sh}(H)$ and their special Green rings, that is , the representation
ring $r(H)$, the derived ring $dr(H)$ and its subring, the shift
ring $sh(H)$.

\subsection{On representation rings of module categories}

Let $H$ be a $K$-Hopf algebra over $K$ with comultiplication
$\Delta$ and counit $\varepsilon$. It is well-known that $_Hmod$ is
a monoidal category with the tensor product $\otimes_K$. Usually,
the Green ring $gr(_Hmod)$ of $_Hmod$ is called the {\em
representation ring} of $H$, denoted by $r(H)$, see \cite{[13]}. In
this case, $r(H)$ is generated by all $[V]$ of indecomposable
$H$-modules $V\in {\rm ind}(H)$ with $[V]+[W]=[V\oplus W]$ and
$[V][W]=[V\otimes_{K}W]$.

Note that

(i)~ the module structure of $V\otimes_{K}W$ is obtained
by $h(v\times w)=\sum\limits_{(h)}(h'v\times h''w)$ for $h\in H,
v\in V, w\in W$;

(ii)~ For the identity $[K]$, $K$ is viewed as an $H$-module by
$h\cdot k= \varepsilon(h) k$ for any $h\in H$, $k\in K$;

(iii)~ The distributive law of $r(H)$ is due to the bilinearity of
$\otimes_K$;

(iv)~ If $H$ is cocommutative, then $r(H)$ is a commutative ring.

\subsection{On derived rings of bounded derived categories}

Recall that in general, for a $K$-algebra $A$, one denotes by
$Ch(A)$ the complex category of $A$,  $K(A)$ the homotopy category
of $A$, $D(A)$ the derived category of $A$ and $D^{b}(A)$ the
bounded derived category of $A$.
 For details on derived categories, see \cite{[14]},\cite{[25]}.

 For a complex $M^{\bullet}=\{M_{i},d_{i}\}$ in these categories, we call $M^{\bullet}$ an {\em $A$-complex} since all $M_i$ are $A$-modules.

 For two indecomposable complexes $M^{\bullet}=\{M_{i},d_{i}\},
N^{\bullet}=\{N_{i},d'_{i}\}$ in $D^{b}(A)$, the tensor product of
chain complexes $M^{\bullet}\otimes N^{\bullet}$ is given in
\cite{[25]}, which is shown as follows:

 Constructing the double complex
$C_{M^{\bullet},N^{\bullet}}=\{M_{i}\otimes N_{j}\}$ together with
the maps
$$d^{h}_{i}=d_{i}\otimes {\rm id}_{N_{j}}: M_{i}\otimes N_{j}\rightarrow M_{i-1}\otimes
N_{j};\;\;\;\;\;\;\;\;\; d^{v}_{j}=(-1)^{i}~{\rm id_{M_{i}}}\otimes
d'_{j}: M_{i}\otimes N_{j}\rightarrow M_{i}\otimes N_{j-1},$$ one
defines its total complex ${\rm
Tot}^{\oplus}(M^{\bullet},N^{\bullet})$ by $({\rm
Tot}^{\oplus}(M^{\bullet},N^{\bullet}))_{n}=\bigoplus\limits_{i+j=n}
M_i\otimes N_j$ with the maps
$\widehat{d_n}=\bigoplus\limits_{i+j=n}(d_{i}^{h}+d_{j}^{v}).$

  We call ${\rm Tot}^{\oplus}(M^{\bullet},N^{\bullet})$ as the {\em tensor product of chain complexes} $M^{\bullet}$ and $N^{\bullet}$, that is,
  define $M^{\bullet}\otimes N^{\bullet}={\rm Tot}^{\oplus}(M^{\bullet},N^{\bullet})$, where $\widehat{d_n}$ is said to be the {\em differentials} of $M^{\bullet}\otimes N^{\bullet}$.

 For an arbitrary algebra $A$, we can only define
$M_i\otimes N_j$ as a left $A$-module via the left $A$-action on
$M_i$, which does not be related with the module structure of $N_j$.
So, we do not think $M^\bullet\otimes N^\bullet$ as complex in
$D^b(A)$, in general.

 Now, we always assume that $A=H$ is a $K$-Hopf algebra with
comultiplication $\Delta$. Then, for two complexes
$M^{\bullet}=\{M_{i},d_{i}\}, N^{\bullet}=\{N_{j},d'_{j}\}\in
D^{b}(H)$, replying on the $H$-module structures of $M_i$ and $N_j$,
we can define canonically an $H$-module structure on all  terms,
that is, for any $a\in H$, $m_{i}\in M_{i}$, $n_{j}\in N_{j}$,
define $a(m_{i}\otimes n_{j})=\sum\limits_{(a)} a'm_{i}\otimes
a''n_{j}$. From this, we have the following:

\begin{Lem}\label{lem2.2}
{\rm Let $M^{\bullet}=\{M_{i},d_{i}\},~
N^{\bullet}=\{N_{i},d'_{i}\}$ be two $H$-complexes in $D^{b}(H)$.
Then the tensor product $M^{\bullet}\otimes N^{\bullet}$ is also an
$H$-complex in $D^{b}(H)$.}
\end{Lem}
\begin{proof}
 We know all $M_i\otimes N_j$ are $H$-modules given above.
So, we need only to claim all differentials
$\widehat{d_n}=\bigoplus\limits_{i+j=n}(d_{i}^{h}+d_{j}^{v})$ are
$H$-module morphisms among $\{M_i\otimes N_j\}$. In particular, it
suffices to prove $d_{i}^{h}+d_{j}^{v}$ are $H$-module morphisms. In
fact, for any $a\in H$, $m_{i}\in M_{i}$, $n_{j}\in N_{j}$,
{\begin{eqnarray*} (d_{i}^{h}+d_{j}^{v})(a(m_{i}\otimes n_{j})) & =
& \sum\limits_{(a)}d_{i}(a'm_{i})\otimes
a''n_{j}+\sum\limits_{(a)}a'm_{i}\otimes
d'_{j}(a''n_{j}) \\
   & = & \sum\limits_{(a)} a'd_{i}(m_{i})\otimes a''n_{j}+\sum\limits_{(a)} a'm_{i}\otimes
a''d'_{j}n_{j} \\
   & = & a((d_{i}^{h}+d_{j}^{v})(m_{i} \otimes n_{j})).
\end{eqnarray*}}
\end{proof}

With this tensor product, the bounded derived category $D^b(H)$ can be viewed as a monoidal category.

Now let us recall that a morphism $f: M^{\bullet}\rightarrow
N^{\bullet}$ is called a \emph{quasi-isomorphism} if it induces
group isomorphisms between all homology groups of $M^{\bullet}$ and
$N^{\bullet}$, written as $M^{\bullet}\stackrel{f}{\cong} N^{\bullet}$.
We denote by $[M^{\bullet}]$ the isomorphism class of a chain
complex $M^{\bullet}$ in $D^{b}(H)$.

The \emph{mapping cone} of $f$, denoted as ${\rm Cone}(f)$,  is
defined as a complex $M^{\bullet}[1]\oplus N^{\bullet} $, with the
differentials \[\left(\begin{array}{cc}
 -d_{i-1} & 0 \\
 f_{i}& d'_{i} \\
\end{array}\right),\;\;\;\; \text{for all}\;\;\; i\in \mathds{Z}.\]

\begin{Lem}\label{lem2.3}
{\rm Let $f: M^{\bullet}\rightarrow M'^{\bullet}$ and $g:
N^{\bullet}\rightarrow N'^{\bullet}$ be two quasi-isomorphisms in
$Ch(H)$, then $f \otimes g: M^{\bullet}\otimes
N^{\bullet}\rightarrow M'^{\bullet}\otimes N'^{\bullet}$ is also a
quasi-isomorphism.}
\end{Lem}

\begin{proof}
 We know that a morphism between two complexes is a quasi-isomorphism if and only if its mapping cone is exact, see \cite{[25]}.
 So, it is enough to prove that the complex ${\rm Cone}(f\otimes g)$ is exact.  Due to symmetry, we only consider the exactness of  ${\rm Cone}(f\otimes {\rm id})$.
But, we have ${\rm Cone}(f\otimes {\rm id})={\rm Cone}(f)\otimes
N^{\bullet}$. Since the double complex are of row and column
bounded, we only need to prove that each row or column is exact by
the Acyclic Assembly Lemma in \cite{[25]}. Now the row exactness of
the double complex {\rm Cone}$(f)\otimes N^{\bullet}$ follows from
the exactness of {\rm Cone}$(f)$.
\end{proof}

\begin{Def}\label{def2.4}
{\rm Let $H$ be a $K$-Hopf algebra. The \emph{derived ring} $dr(H)$
of a $K$-Hopf algebra $H$ is defined as the Green ring $gr(D^b(H))$
of the monoidal category $D^b(H)$. Concretely, for an bounded
$H$-complex $M^{\bullet}$ in $D^b(H)$, we denote by $[M^{\bullet}]$
the isomorphism class of $M^{\bullet}$.

(1)~ $dr(H)$ is the abelian group generated by all  isomorphism
classes of bounded complexes in $D^{b}(H)$ with addition defined by
the relations $[M^{\bullet}]+[N^{\bullet}]=[M^{\bullet}\oplus
N^{\bullet}]$
 for any bounded complexes $M^\bullet, N^\bullet$ in
$D^{b}(H)$;

(2)~ In $dr(H)$, the multiplication is defined by the relations
$[M^{\bullet}][N^{\bullet}]=[M^{\bullet}\otimes N^{\bullet}]$.}
\end{Def}

Note that

(i)~  The multiplication of $dr(H)$ is
well-defined due to Lemma \ref{lem2.3};

(ii)~ $[K^\bullet]$ is  the identity of $dr(H)$, since
$K^{\bullet}\otimes M^{\bullet}\cong M^{\bullet}\otimes
K^{\bullet}\cong M^{\bullet}$ for any $M^{\bullet}\in D^{b}(H)$,
where $K^\bullet$ is the stalk complex of the trivial $H$-module $K$
at the $0$-position;

(iii)~ The associativity of $dr(H)$ holds since  for any
$L^{\bullet}, M^{\bullet}, N^{\bullet}$ in $D^b(H)$,
$$([L^{\bullet}][M^{\bullet}])[N^{\bullet}]=([L^{\bullet}\otimes
M^{\bullet}])[N^{\bullet}]= [L^{\bullet}\otimes M^{\bullet}\otimes
N^{\bullet}]=[L^{\bullet}]([M^{\bullet}\otimes
N^{\bullet}])=[L^{\bullet}]([M^{\bullet}][N^{\bullet}]).$$

(iv)~ The distributive law of $dr(H)$ holds since
$(L^{\bullet}\oplus M^{\bullet})\otimes N^{\bullet} \cong
(L^{\bullet}\otimes N^{\bullet})\oplus(M^{\bullet}\otimes
N^{\bullet});$

(v)~ The abelian group $dr(H)$ is free with a $\mathds{Z}$-basis
$\{[M^{\bullet}]~|~M^{\bullet}\in {\rm ind}(D^{b}(H))\}$ by the
Krull-Schmidt property of $D^{b}(H)$;

(vi)~ If $H$ is cocommutative, then $dr(H)$ is a commutative ring.

\subsection{On shift rings of (derived) shift categories}

For any indecomposable module $M\in {\rm ind}(_Hmod)$,  its stalk
complex $M^{\bullet}$ and shift objects $M^{\bullet}[n], n\in
\mathds{Z}$ are indecomposable complexes in $Ch(H)$ and $D^{b}(H)$, where $[1]$
means the shift functor. It is clear that for all $M, N\in {\rm
ind}(_Hmod),i,j\in \mathds{Z}$,  in $Ch(H)$ and $D^{b}(H)$, we have
 \begin{equation}\label{addshift}
M^{\bullet}[i]\otimes N^{\bullet}[j]\cong (M^{\bullet}\otimes
N^{\bullet})[i+j]=(M\otimes N)^{\bullet}[i+j].
\end{equation}

 Actually, it is not difficult to understand that the isomorphism in (\ref{addshift}) on stalk complexes can be lifted to any indecomposable objects in  $D^{b}(H)$, that is, for any $M^{\bullet}, N^{\bullet}\in {\rm ind}D^{b}(H)$, $m, n\in \mathbb{Z}$, the isomorphism holds:
\begin{equation}\label{addcomplex}
M^{\bullet}[m]\otimes N^{\bullet}[n]\cong (M^{\bullet}\otimes
N^{\bullet})[m+n].
\end{equation}

Define $Ch^{sh}(H)$ to be the full subcategory of $Ch(H)$ whose
objects are all $M^{\bullet}[n], \forall M\in _Hmod,
n\in\mathds{Z}$. $M^{\bullet}[n]$ is indecomposable in $Ch^{sh}(H)$
if and only if $M\in {\rm ind}(_Hmod)$.  Due to trivially
$M^{\bullet}\otimes N^{\bullet}\cong (M\otimes N)^{\bullet}$ in
$Ch(H)$ and that $Ch^{sh}(H)$ is closed under $\otimes$ by
(\ref{addshift}), we see that $Ch^{sh}(H)$ is a monoidal subcategory
of $Ch(H)$ on the same $\otimes$.

We call the monoidal category $Ch^{sh}(H)$ the {\em shift
category} of $H$, whose Green ring is said to be the {\em shift
ring} of $H$, denoted as $sh(H)$.

Analogously, we define $D^{sh}(H)$ to be the full subcategory of
$D^b(H)$ whose objects are $M^{\bullet}[n]$ for any $M\in _Hmod,
n\in\mathds{Z}$, which is called as the {\em derived shift category}
of $H$. Also, $D^{sh}(H)$ is a monoidal
category.

${\rm Hom}_{D^{b}(H)}(M^{\bullet}[n], M^{\bullet}[n])$ is local when $M$
is an indecomposable $H$-module, since
$${\rm Hom}_{D^{b}(H)}(M^{\bullet}[n], M^{\bullet}[n])\cong {\rm Hom}_{H}(M,
M).$$ Therefore, the indecomposable objects in $D^{sh}(H)$ are all
$M^{\bullet}[n]$ for $M\in {\rm ind}_Hmod, n\in\mathds{Z}$.

Hence, from above, indecomposable objects of $Ch^{sh}(H)$ are the
same with that of $D^{sh}(H)$ , that is,  all $M^{\bullet}[n]$ for
$M\in {\rm ind}(_Hmod), n\in\mathds{Z}$.
 It follows that the Green ring of $D^{sh}(H)$ is the same with that of
$Ch^{sh}(H)$, that is, the shift ring $sh(H)$.

 According to the definition, the following facts are
obtained:
\begin{Fact}\label{fact1}
{\rm Let $H$ be a $K$-Hopf algebra. Then,

(i)~ The shift ring $sh(H)$ is a subring of the derived ring $dr(H)$ generated by $[M^{\bullet}[n]], \forall M\in {\rm ind}(_Hmod),\;\; n\in \mathds{Z}$;

(ii)~ $sh(H)$ is a free abelian group with basis
$[M^{\bullet}[n]]$ for $M\in {\rm ind}(_Hmod),~ n\in \mathds{Z}$;

(iii)~ The shift ring $sh(H)$ has a $\mathds Z$-graded
structure, more precisely, for $n\in \mathds{Z}$, whose component of
$n$-degree $sh(H)_{(n)}$ is the free abelian subgroup generated by
$[M^{\bullet}[n]]$ for all $M\in {\rm ind}(_Hmod)$;

(iv)~ $sh(H)_{(i)}sh(H)_{(j)}=sh(H)_{(i+j)}$ for any $i,j\in\mathds{Z}$;

(v)~ The representation ring $r(H)$ is indeed the $0$-component of $sh(H)$, that is, $r(H)\cong sh(H)_{(0)}$. }
\end{Fact}

In these facts, (iv) is from the isomorphism (\ref{addshift}) and (v) can be easily seen from that $[M^{\bullet}]=[M^{\bullet}[0]]$ for all $M\in {\rm ind}(_Hmod)$.

Now we can give the connection among $r(H)$, $sh(H)$, $dr(H)$ as
visualised in Figure $1$.
\begin{figure}[h] \centering
  \includegraphics*[33,704][461,749]{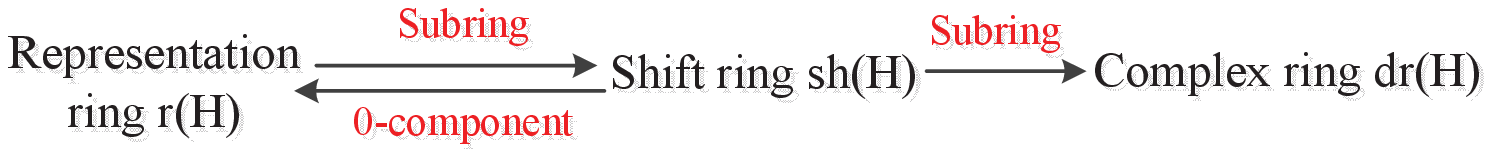}

 {\rm Figure 1 }
\end{figure}
\bigskip

\subsection {Polynomial characterizations }
\label{prop2.2}~
 Throughout this part, $H$ refers a representation-finite $K$-Hopf algebra.

 Up to isomorphism, let $\mathcal M=\{M_1,\cdots,M_t\}$ be the set of all $t$ indecomposable $H$-modules. In particular, say $M_1=K$. The representation ring $r(H)$ can be realized as a quotient of a certain finitely generated free ring, that is, for the \emph{free algebra} $K\langle \mathbf X\rangle$ on a set of indeterminates   $\mathbf X=\{x_{1}, x_{2},\cdots, x_{t}\}$,  we can define an epimorphism of rings $\varphi: K\langle \mathbf X\rangle\rightarrow
r(H)$ satisfying $\varphi(1)=[K]$, $\varphi(X_{i})=[M_{i}]$ for $i=1,\cdots,t$.

  Then for any $1\leq i,j\leq t$, there are non-negative integers $k_{i,j}^{l} (i\leq l\leq t)$ appearing as the structure constants of the representation ring $r(H)$ satisfying
  \begin{equation}\label{strnum}
  [M_{i}][M_{j}]=\sum_{1\leq l\leq t} k_{i,j}^{l}[M_{l}].
   \end{equation}

   Define an ideal $I$ of $K\langle\mathbf X\rangle$ generated by $\{x_{i}x_{j}-\sum_{1\leq l\leq t} k_{i,j}^{l} x_{l}~ |~ 1\leq i,j\leq
t\}$, then the representation ring $r(H)$ can be realized as $$r(H)\cong K\langle \mathbf X\rangle/(I,1-x_{1}).$$

 Moreover, if $H$ is cocommutative, then $r(H)$ becomes a commutative ring due to the definition of representation ring. In this case, the free algebra $K\langle \mathbf X\rangle$ is replaced by the polynomial $K[\mathbf X]$ through the relations $[M_i][M_j]=[M_j][M_i]$, that is, $$r(H)\cong K[\mathbf X]/(I,1-x_{1}).$$

However, in general, using some relations among the $t$ indecomposable modules of a representation-finite $K$-Hopf algebra $H$, it is possible for one to find a positive integer $s$ with $s<t$ and $s$ generators consisting of iso-classes of indecomposable modules to generate $r(H)$ with more relations. Then we will be able to find  $\mathbf Y=\{y_{1},\cdots,y_{s}\}$ so as to get an epimorphism  $K[Y]\rightarrow r(H)$. This is just that we will do in the following sections for some Nakayama truncated algebras $H=KZ_{n}/J^{d}$.

Following  the above discussion, we can give the further conclusion on shift rings and derived rings as follows.
\begin{Thm}\label{thm10.2}
{\rm For a representation-finite $K$-Hopf algebra $H$, assume the number of the indecomposable modules of $H$ is $t$ up to isomorphism. Then, the following statements hold:

(i)~ The shift ring $sh(H)\cong K\langle \mathbf X'\rangle/(I', 1-X_{1})$, where $$\mathbf X'= \{X_{1}[i_{1}],X_{2}[i_{2}],\cdots,X_{t}[i_{t}]~ |~ i_{1},\cdots,i_{t}\in \mathds{Z}\}$$ is the set of the indeterminates,  and in particular, denote $X_{k}\equiv X_{k}[0], \forall 1\leq k\leq t$, the ideal $I'$ is generated by $$\{X_{i}[r]X_{j}[s]-\sum_{1\leq l\leq t} k_{i,j}^{l}X_{l}[r+s]~ |~ 1\leq i,j\leq t, r,s\in \mathds{Z}\}$$ with $k_{i,j}^{l} (1\leq i,j,l\leq t)$ the structure constants of $r(H)$ in (\ref{strnum}).

(ii)~ Let  $\{N_\lambda, \lambda\in\Lambda\}$ be the set of representatives of the orbits of indecomposable complexes in $D^{b}(H)$ under the shift functor $[1]$. Then the derived ring $dr(H)\cong K\langle\mathbf Y'\rangle/(J', 1-Y_{1})$, where $\mathbf Y'=\{Y_{\lambda}[s]~ |~ \lambda\in \Lambda, s\in \mathds{Z}\}$ is the set of the indeterminates, and in particular, $Y_1$ a fixed indeterminate with $1\in \Lambda$, denote $Y_{\lambda}\equiv Y_{\lambda}[0], \forall \lambda\in \Lambda$, the ideal $J'$ is generated by $$\{Y_{\lambda}[s]Y_{\mu}[t]-\sum_{\eta\in\Lambda, r\in\mathds Z} k_{\lambda,\mu}^{\eta,r}Y_{\eta}[r+s+t]~ |~ \lambda,\mu\in \Lambda, s,t\in \mathds{Z}\},$$ and  $k_{\lambda,\mu}^{\eta,r}$ is the multiplicity of $N_{\eta}[r]$ in the decomposition of $N_{\lambda}\otimes N_{\mu}$ for $\lambda,\mu,\eta\in \Lambda, r\in \mathds{Z}$.

(iii)~ Furthermore, when $H$ is a cocommutative, then $sh(H)\cong K[\mathbf X']/(I', 1-X_{1})$ and $dr(H)\cong K[\mathbf Y']/(J', 1-Y_{1})$. }
\end{Thm}

\begin{proof}
For (i), we define a ring homomorphism $\varphi: K\langle X'\rangle \rightarrow sh(H)$ such that $\varphi(X_{i}[s])=[M_{i}^{\bullet}[s]]$ for any $1\leq i\leq t$, $s\in \mathds{Z}$ and $\varphi(1)=[M_{1}^{\bullet}[0]]$. It is clear that ${\rm
ker}\varphi$ is an ideal of $K\langle X'\rangle$  generated by $I'$ and $1-X_{1}[0]$, then (i) follows at once, by (\ref{addshift}).

 The proof of (ii) is similar to that of (i)  by (\ref{addcomplex}), only note to define a ring homomorphism $\psi: K\langle\mathbf Y'\rangle\rightarrow dr(H)$ such that $\psi(Y_{\lambda}[s])=[N_{\lambda}[s]]$ for any $\lambda\in\Lambda, s\in \mathds{Z}$.

Now (iii) follows from (i),(ii) and the fact that the cocommutativity of $H$ implies the commutativity of $dr(H)$ and $sh(H)$ directly.
\end{proof}

\begin{Rem}
 {\rm In Theorem \ref{thm10.2} (i), the structure constants of the shift ring $sh(H)$ are the same ones with that of the representation ring $r(H)$. The similar matter happens for the derived ring $dr(H)$ in Theorem \ref{thm10.2} (ii), that is, the structure constants $k_{\lambda,\mu}^{\eta,r}$ in $Y_{\lambda}[s]Y_{\mu}[t]=\sum_{\eta\in\Lambda, r\in\mathds Z} k_{\lambda,\mu}^{\eta,r}Y_{\eta}[r+s+t]$ are the same with that in $Y_{\lambda}Y_{\mu}=\sum_{\eta\in\Lambda, r\in\mathds Z} k_{\lambda,\mu}^{\eta,r}Y_{\eta}[r]$.}
\end{Rem}

 Starting from the next section, we will characterize the representation rings of a class of Hopf algebras and further will embed them into the corresponding shift rings and derived rings. For this class of cocommutative Hopf algebras, i.e. the Nakayama truncated algebras,  the polynomial characterizations of their representation rings, shift rings and derived rings will be given by less numbers of indeterminates.

\bigskip
\section{  Hopf algebra structure of Nakayama truncated algebras}

\subsection{ Definition and related indecomposable modules}

The Nakayama truncated algebras, or say, the truncated cycle algebras,
is coming from \cite{[2]}. More precisely, let $Z_{n}$ be the oriented
cycle of length $n$, that is, which has $n$ vertices
$\{v_{\overline{0}},\cdots,v_{\overline{n-1}}\}$ and $n$ arrows
$\{\alpha_{\overline{0}},\cdots,\alpha_{\overline{n-1}}\}$ such that
the origin vertex $s(\alpha_{\overline{i}})$ of the arrow
$\alpha_{\overline{i}}$ equals the terminus vertex
$t(\alpha_{\overline{i-1}})$ of $\alpha_{\overline{i-1}}$.
 Here the below indices of vertices and arrows in $Z_{n}$ are
denoted by using
$\{\overline{0},\overline{1},\cdots,\overline{n-1}\}$ the set of
elements of the residue class of the abelian group
$\mathbb{Z}/n\mathbb{Z}$, since in the sequel we will correspond
the set of the vertices of $Z_{n}$ to the set of the elements of group $G=\mathbb{Z}/n\mathbb{Z}$, see Proposition \ref{prop3.2}. Let $J$ denote the two-sided
ideal of path algebra $KZ_{n}$ generated by all arrows.

The {\em Nakayama truncated algebras} are defined as $KZ_{n}/J^{d}$ for any positive integers $d$, which are called the {\em truncated quotients} of the path algebras $KZ_{n}$ for any $n$.

 For $A=KZ_{n}/J^{d}$, by Theorem V.3.5 in \cite{[1]}, for any indecomposable
$A$-module $M$, there exists an indecomposable projective $A$-module
$P$ and an integer $t$ with $1\leq t\leq d$ such that $M\cong P/{\rm
rad}^{t}P$. Thus there are totally $nd$ indecomposable $A$-modules
$$M(i,\overline{j})=P_{\overline{j}}/{\rm rad}^{i} P_{\overline{j}}, 1\leq i\leq d$$
where $P_{\overline{j}}$ is the indecomposable projective $A$-module
at the vertex $v_{\overline{j}}$. Denote by $S_{\overline{j}}$ the
simple $A-$ module at $v_{\overline{j}}$, notice that
$M(1,\overline{j})=S_{\overline{j}}$ and
$M(d,\overline{j})=P_{\overline{j}}$.

\subsection{ Hopf algebra structure}

It has already known that the Nakayama truncated algebras $KZ_{n}/J^{d}$ are of representation-finite. Now we will give the sufficient and necessary condition for a Nakayama truncated algebra to be a Hopf algebra over a field $K$ of char$K=p$ a prime.

First recall the famous description of the ordinary quiver (or say, Ext-quiver) of a finite dimensional basic Hopf algebra in \cite{[12]}. Let $H$ be a finite dimensional basic Hopf
algebra over $K$, then there exists a finite group $G$ and a {\em weight sequence}
$W=(w_{1},w_{2},\cdots,w_{r})$ of $G$ (i.e., for each $g\in G$, the sequence $W$ and
$(gw_{1}g^{-1},gw_{2}g^{-1},\cdots,gw_{r}g^{-1})$ are the same up to a permutation), such
that
$$H\cong K\Gamma_{G}(W)/I$$ as Hopf algebras for an admissible ideal $I$, where the
quiver $\Gamma_{G}(W)$ is defined
as follows: the vertice set $\Gamma_{G}(W)_{0}=\{v_{g}\}_{g\in G}$ and the arrow set
$\Gamma_{G}(W)_{1}=\{(a_{i},g):v_{g^{-1}}\rightarrow v_{w_{i}g^{-1}}|g\in G,w_{i}\in W\}$, which is just the ordinary quiver of $H$.
We call
this quiver $\Gamma_{G}(W)$ the {\em covering quiver} of $H$ with respect to the weight sequence $W$.

On the contrary, given a finite group $G$ and a weight sequence $W$ of $G$, the Hopf algebra
structure on the path algebra $K\Gamma_{G}(W)$ are introduced in \cite{[12]}. More
precisely, the
comultiplication $\Delta$, the counit $\varepsilon$ and the antipode $S$ are defined
as
follows:
\[\begin{array}{ccc}
\Delta(v_{h})=\sum\limits_{g\in G} v_{hg^{-1}}\otimes v_{g}, &
\varepsilon(v_{h})=\left\{\begin{array}{ll}
1, &\mbox{$h=e$} \\
0, &\mbox{$h\neq e$}
\end{array}\right., &
S(v_{h})=v_{h^{-1}}
\end{array}\]
\[\begin{array}{ccc}
\Delta(\alpha)=\sum\limits_{g\in G}(g\cdot \alpha\otimes v_{g}+v_{g}\otimes \alpha \cdot g),
&
\varepsilon(\alpha)=0,  &
S(\alpha)=-f\cdot \alpha \cdot d
\end{array}\]
and then are extended linearly to be an algebra
$($resp. anti-algebra$)$ morphism respectively
such that $K\Gamma_{G}(W)$ becomes a Hopf algebra,
where $g\cdot \alpha$, $\alpha \cdot g$, $f \cdot \alpha \cdot d$ are induced by the
allowable $KG$-bimodule structure on $K\Gamma_{G}(W)$.

However, in general, $K\Gamma_{G}(W)$ is an infinite dimensional algebra. In order to
obtain finite dimensional Hopf algebras as quotients of the path algebras $K\Gamma_{G}(W)$,
Green
and Solberg constructed two ideals $I_{p}$ and $I_{q}$ and gave the condition
for the ideal
$(I_{p},I_{q})$ generated by $I_{p},I_{q}$ to be a Hopf ideal of $K\Gamma_{G}(W)$.

Actually, let $\alpha:v_{d}\rightarrow v_{f}$, $\beta:v_{d}\rightarrow v_{h}$ be arrows in $\Gamma_{G}(W)$ with the same starting vertex, $f,d,h\in G$. Define
$$r(\alpha)=d^{-1}f, \;\;\;\;\;\;\;\;\; l(\alpha)=fd^{-1}, \;\;\;\;\;\;\;\;\; q(\alpha,\beta)=((r(\alpha))^{-1}\cdot
\beta)\alpha-(\alpha\cdot (l(\beta))^{-1})\beta.$$

On one hand, let $I_{q}$ be the ideal in
$K\Gamma_{G}(W)$
generated by the elements $q(\alpha,\beta)\cdot g$ for all such
$\alpha$ and
$\beta$ as above with  $\alpha \neq \beta$ and all $g\in G$.

On the other hand, for each arrow $\alpha$ in $\Gamma_{G}(W)$, choose a positive integer $m_{\alpha}\geq 2$
so that $m_{\alpha}=m_{g\cdot \alpha}$ for all $g$ in $G$. Denote by $p(\alpha)$ the path defined as follows:
$$ p(\alpha)=\left\{\begin{array}{ll}
{\alpha}^{m_{\alpha}},  &\mbox{ if $\alpha$ is a loop;} \\
\prod\limits_{i=m_{\alpha}-1}^0 ((r(\alpha))^{-i}\cdot \alpha), &\mbox{ otherwise.}
\end{array}\right.$$
Let $I_{p}$ be the ideal generated by the elements $p(\alpha)\cdot g$ for all
arrows $\alpha$ in $\Gamma_{G}(W)$ and $g\in G$. For more details, see \cite{[12]}.
\begin{Thm}{\rm (Corollary 5.4, \cite{[12]}})\label{thm2.1}
{\rm Let $G$ be a finite group and $W=(w_{1},w_{2},\cdots,w_{n})$ a weight sequence with
$n\geq 1$ such that $K\Gamma_{G}(W)$ becomes a Hopf algebra whose Hopf structure given by
an allowable $KG$-bimodule structure on $K\Gamma_{G}(W)$.
Let $I_{p}$ and $I_{q}$ be the ideals in $K\Gamma_{G}(W)$ defined above. Suppose that

(i)~ The subgroup of $G$ generated by the elements of $W$ is abelian;

(ii)~ For all arrows $\alpha$ and $\beta$ starting in the same vertex, there exists
$0\not=c_{\beta}(\alpha)\in K$ such that $\alpha \cdot
l(\beta)=c_{\beta}(\alpha)r(\beta)\cdot
\alpha$ and $c_{\alpha}(\beta)=c_{\beta}(\alpha)^{-1}$ in the case $\alpha \neq \beta$;

(iii)~ $\sum\limits_{\sigma \in T_{s}(m_{\alpha})} \prod\limits_{i\notin \sigma}
c_{\alpha}(\alpha)^{-i}=0$ for all $s=1,\cdots, m_{\alpha}-1$, non-loop arrows $\alpha$ in
$\Gamma_{G}(W)$, where $T_{s}(n)$ denotes the set of all subsets of $\{0,1,\cdots,
n-1\}$ consisting of $s$ elements;

(iv)~ $\binom{m_{\alpha}}{i}=0$ in $K$ for all $i=1,2,\cdots, m_{\alpha}-1$, loops $\alpha$ in $\Gamma_{G}(W)$.
\\
Then $(I_{p},I_{q})$ is a Hopf ideal in $K\Gamma_{G}(W)$ and  the obtained quotient
$K\Gamma_{G}(W)/(I_{q},I_{p})$ is a finite dimensional Hopf algebra.}
\end{Thm}

 Now, we can give the sufficient and necessary condition for a Nakayama truncated algebra to be a Hopf algebra with Hopf structure given in Theorem \ref{thm2.1} as follows.

\begin{Prop}\label{prop3.2}
{\rm For a Nakayama truncated algebra $KZ_{n}/J^{d}$ over a field
$K$ of characteristic $p$, let
$G=\mathbb{Z}/n\mathbb{Z}=\{\overline{0},\overline{1},\cdots,\overline{n-1}\}$ be the
residue class group modulo $n$ with a weight sequence
$W=\{\overline{1}\}$. Then,

(i)~ the covering quiver $\Gamma_{G}(W)$ is just the $n$-th oriented cycle $Z_{n}$;

(ii)~ $KZ_{n}/J^{d}$ is a cocommutative Hopf algebra about the covering quiver $Z_n$
 with Hopf structure given in Theorem \ref{thm2.1} if
and only if $d=p^m\leq n$ for some $m>0$. }
\end{Prop}

\begin{proof}
(i)~ It follows directly from the definition of $\Gamma_{G}(W)$.

(ii)~ {\bf ``Only if":~~ } In the oriented cycle
$\Gamma_{G}(W)=Z_{n}$, the vertex set $(Z_{n})_{0}=\{v_{g}\}_{g\in
G}$ and the arrow set $(Z_{n})_{1}=\{\alpha_{g}:v_{g}\rightarrow
v_{\overline{1}+g}~|~g\in G\}$. Then the comultiplication, counit,
antipode of the Hopf structure of $KZ_n$ are defined respectively as
follows:
\[\begin{array}{ccc}
\Delta(v_{h})=\sum\limits_{g_{1}+g_{2}=h,~ g_{1},g_2\in G} v_{g_{1}}\otimes v_{g_{2}},
&
\Delta(\alpha_{h})=\sum\limits_{g_{1}+g_{2}=h,~ g_{1},g_2\in G}(\alpha_{g_{1}}\otimes
v_{g_{2}}+v_{g_{1}}\otimes \alpha_{g_{2}})
\end{array}\]

Since $Z_{n}$ has no loop, we have $I_{q}=0$.

And, for each $\alpha$ in $Z_{n}$, choose $m_{\alpha}=d$. Hence we obtain that $I_{p}$ is the ideal generated by the elements $(\prod\limits_{i=d-1}^0 ((r(\alpha))^{-i}\cdot \alpha))\cdot g$ for all arrows $\alpha$ in $Z_{n}$ and $g\in G$, that is, $I_{p}=J^{d}$.

Therefore, $(I_{p},I_{q})=J^d$ and then by Theorem \ref{thm2.1}, which is a Hopf ideal in $K\Gamma_{G}(W)$ if and only if
$\sum\limits_{\sigma \in T_{s}(d)} \prod\limits_{i\notin \sigma} 1^{-i}=0$ for all $s=1,\cdots, d-1$ and all arrows $\alpha$ in
$\Gamma_{G}(W)$, that is,
\begin{equation}\label{char}
\binom{d}{i}=0~~ \text{in}~~ K,~~ \forall 1\leq i\leq d-1
\end{equation}
since $\sum\limits_{\sigma \in T_{s}(d)}
\prod\limits_{i\notin \sigma} 1^{-i}=\binom{d}{i}$.

Now by the fact in \cite{[22]} that ${\rm gcd}(\binom{d}{1},\binom{d}{2},\cdots,\binom{d}{d-1})=
\left\{\begin{array}{ll}
p, &\mbox{$d$ {\rm is the power of prime} $p$;} \\
1, &\mbox{{\rm otherwise,}}
\end{array}\right.$
 we have $d=p^m$ for some prime $p$, $m>0$. The conclusion follows at once.

{\bf ``If":~~} If $d=p^m\leq n$ holds for some $m>0$, then the formula (\ref{char}) holds and thus,
 $(I_{p},I_{q})=J^d$ is a Hopf ideal in $K\Gamma_{G}(W)$ such that $KZ_n/J^d$ becomes a Hopf algebra with Hopf structure in Theorem \ref{thm2.1}, whose  cocommutativity follows directly from the definition of the comultiplication $\Delta$.
\end{proof}

\bigskip
\section{{ Elementary Lemma }}

First we may fix some notations, denote the basis of
$M(i,\overline{j})$ as $\{v_{\overline{j}},
\alpha_{\overline{j}},\alpha_{\overline{j+1}}\alpha_{\overline{j}},\cdots,
\alpha_{\overline{j+i-2}}\cdots
\alpha_{\overline{j+1}}\alpha_{\overline{j}}\}$, and abbreviated
them as $v_{\overline{j}}^{0}=v_{\overline{j}}$,
$v_{\overline{j}}^{1}=\alpha_{\overline{j}}$,$\cdots$,
$v_{\overline{j}}^{\overline{i-1}}=\alpha_{\overline{j+i-2}}\cdots
\alpha_{\overline{j+1}}\alpha_{\overline{j}}$, respectively. The
following Figure 2 may be more directly used to express the module
$M(i,\overline{j})$.

\begin{figure}[h] \centering
  \includegraphics*[109,651][451,767]{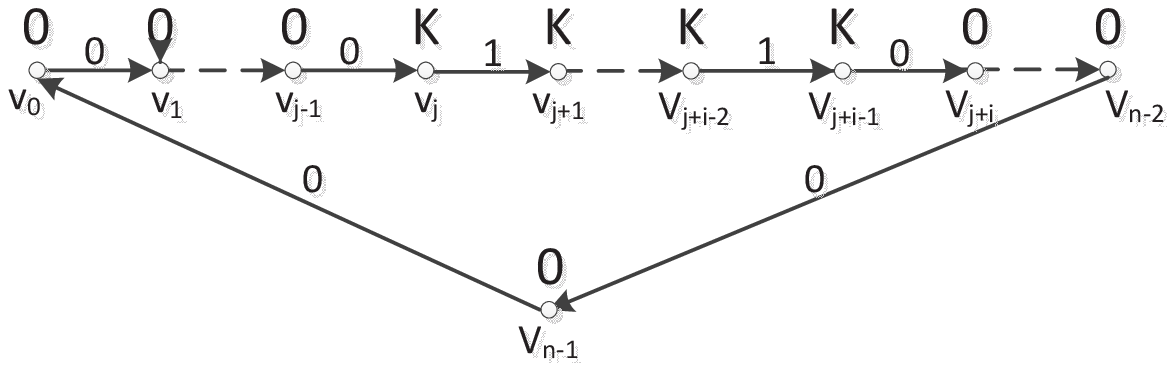}

 {\rm Figure 2 }
\end{figure}

From now on, we always assume that $d=p^m$ for some prime $p$, $m>0$
and ${\rm char} K=p$; ${\rm dim}$,$\otimes$ and ${\rm Hom}$ stand
for ${\rm dim}_{K}$, $\otimes_{K}$ and ${\rm Hom}_{K}$,
respectively. All modules stand for $(KZ_{n}/J^{d})$-modules. Denote
{\bf dim} $X$ as the dimension vector of a $A$-module $X$ and $({\bf
dim} X)_{l}$ as its component at the vertex $v_{\overline{l}}$. The
following two lemmas are easily from a straightforward verification.

\begin{Lem}\label{lem3.1}
{\rm For all $1\leq i,i'\leq d, 0\leq j,j'\leq n-1$, $({\bf dim}
(M(i,\overline{j})\otimes M(i',\overline{j'})))_{\overline{l}}$ is
equal to the number of the partitions of
$\overline{l}=\overline{l_{1}}+\overline{l_{2}}$ under the condition
$({\bf dim} M(i,\overline{j}))_{\overline{l_{1}}}\neq 0$ and $({\bf
dim} M(i',\overline{j'}))_{\overline{l_{2}}}\neq 0$. }
\end{Lem}

\begin{proof}
Choose the basis of $M(i,\overline{j})$ and $M(i',\overline{j'})$ as
$\{v_{\overline{j}}^{\overline{0}},v_{\overline{j}}^{\overline{1}},\cdots,
v_{\overline{j}}^{\overline{i-1}}\}$ and
$\{v_{\overline{j'}}^{\overline{0}},v_{\overline{j'}}^{\overline{1}},\cdots,
v_{\overline{j'}}^{\overline{i'-1}}\}$, respectively, then the basis
of $M(i,\overline{j})\otimes M(i',\overline{j'})$ is
$\{v_{\overline{j}}^{\overline{k}} \otimes
v_{\overline{j'}}^{\overline{k'}}|0\leq k\leq i-1,0\leq k'\leq
i'-1\}$.

Additionally, $({\bf dim} (M(i,\overline{j})\otimes
M(i',\overline{j'})))_{\overline{l}}={\rm dim}
(v_{\overline{l}}\cdot (M(i,\overline{j})\otimes
M(i',\overline{j'})))$, and

 \[\begin{array}{ccl} v_{\overline{l}}\cdot
(v_{\overline{j}}^{\overline{k}}\otimes
v_{\overline{j'}}^{\overline{k'}}) & = &
\Delta(v_{\overline{l}})(v_{\overline{j}}^{\overline{k}}\otimes
v_{\overline{j'}}^{\overline{k'}})=
\sum\limits_{g_{1}+g_{2}=\overline{l}, g_{1}, g_{2}\in G}
(v_{g_{1}}\otimes v_{g_{2}})(v_{\overline{j}}^{\overline{k}}\otimes
v_{\overline{j'}}^{\overline{k'}}) \\
  & = & \sum\limits_{g_{1}+g_{2}=\overline{l}, g_{1},g_{2}\in G} \delta_{g_{1},\overline{j+k}}
  \delta_{g_{2},\overline{j'+k'}}(v_{\overline{j}}^{\overline{k}}\otimes v_{\overline{j'}}^{\overline{k'}}) \\
  & = & \left\{\begin{array}{ll}
v_{\overline{j}}^{\overline{k}}\otimes v_{\overline{j'}}^{\overline{k'}}, &\mbox{if $g_{1}=\overline{j+k},g_{2}=\overline{j'+k'}$} \\
0, &\mbox{otherwise}
\end{array}\right.
\end{array}\]

where $\delta_{i,j}$ is the Kronecker symbol. Thus \\$({\bf dim}
M(i,\overline{j})\otimes M(i',\overline{j'}))_{\overline{l}}=\sharp
\{$the partition of
$l=j+k+j'+k'$ such that $0\leq k\leq i-1, 0\leq k'\leq i'-1$$\}$.

Denote $l_{1}=j+k$, $l_{2}=j'+k'$.  By the definition of
$M(i,\overline{j})$, $({\bf dim}
M(i,\overline{j}))_{\overline{l_{1}}}\neq 0$ if and only if $j\leq
l_{1}\leq j+i-1$, $({\bf dim}
M(i',\overline{j'}))_{\overline{l_{2}}}\neq 0$ if and only if
$j'\leq l_{2}\leq j'+i'-1$. Then we have
$l=(j+k)+(j'+k')=l_{1}+l_{2}$.

Thus, the number of the partition of $l=j+k+j'+k'$ with $0\leq k\leq
i-1, 0\leq k'\leq i'-1$ is equal to the number of the partition of
$l=l_{1}+l_{2}$ with $j\leq l_{1}\leq j+i-1, j'\leq k'\leq j'+i'-1$
and moreover, is equal to the partitions of $l=l_{1}+l_{2}$
satisfying $({\bf dim} M(i,\overline{j}))_{\overline{l_{1}}}\neq 0$
and $({\bf dim} M(i',\overline{j'}))_{\overline{l_{}}}\neq 0\}$.
\end{proof}

Following this result, in the case that $({\bf dim}
M(i,\overline{j})\otimes M(i',\overline{j'}))_{\overline{l}}\neq 0$,
the basis of $(M(i,\overline{j})\otimes
M(i',\overline{j'}))_{\overline{l}}$ is that
$$\{v_{\overline{j}}^{\overline{l_{1}-j}}\otimes v_{\overline{j'}}^{\overline{l_{2}-j'}}~|~l_{1}+l_{2}=l,({\bf dim} M(i,\overline{j}))_{\overline{l_{1}}}\neq 0,
({\bf dim} M(i',\overline{j'}))_{\overline{l_{2}}}\neq 0\}.$$

\begin{Lem}\label{lem3.2}
{\rm If $({\bf dim} M(i,\overline{j}))_{\overline{l_{1}}}\neq 0$,
$({\bf dim} M(i',\overline{j'}))_{\overline{l_{2}}}\neq 0$ and
$l=l_{1}+l_{2}$, then in the module $M(i,\overline{j})\otimes
M(i',\overline{j'})$, it holds that
$$\alpha_{\overline{l}}\cdot (v_{\overline{j}}^{\overline{l_{1}-j}}\otimes
v_{\overline{j'}}^{\overline{l_{2}-j'}})=\alpha_{\overline{l_{1}}}(v_{\overline{j}}^{\overline{l_{1}-j}})\otimes
v_{\overline{j'}}^{\overline{l_{2}-j'}}+v_{\overline{j}}^{\overline{l_{1}-j}}\otimes
\alpha_{\overline{l_{2}}}(v_{\overline{j'}}^{\overline{l_{2}-j'}}).$$
}
\end{Lem}
\begin{proof}
\[\begin{array}{ccl}
\alpha_{\overline{l}}\cdot
(v_{\overline{j}}^{\overline{l_{1}-j}}\otimes
v_{\overline{j'}}^{\overline{l_{2}-j'}}) & = &
\Delta(\alpha_{\overline{l}})(v_{\overline{j}}^{\overline{l_{1}-j}}\otimes v_{\overline{j'}}^{\overline{l_{2}-j'}}) \\
 & = & \sum\limits_{g_{1}+g_{2}=\overline{l},  g_{1},g_2\in G}(\alpha_{g_{1}}\otimes
 v_{g_{2}}+v_{g_{1}}\otimes \alpha_{g_{2}})(v_{\overline{j}}^{\overline{l_{1}-j}}\otimes v_{\overline{j'}}^{\overline{l_{2}-j'}}) \\
 & = & \sum\limits_{g_{1}+g_{2}=\overline{l},  g_{1},g_2\in G}
 \delta_{g_{1},\overline{l_{1}}}\delta_{g_{2},\overline{l_{2}}}(\alpha_{\overline{l_{1}}}(v_{\overline{j}}^{\overline{l_{1}-j}})\otimes
 v_{\overline{j'}}^{\overline{l_{2}-j'}}+ v_{\overline{j}}^{\overline{l_{1}-j}}\otimes \alpha_{g_{2}}(v_{\overline{j'}}^{\overline{l_{2}-j'}})) \\
 & = & \alpha_{\overline{l_{1}}}(v_{\overline{j}}^{\overline{l_{1}-j}})\otimes v_{\overline{j'}}^{\overline{l_{2}-j'}}+v_{\overline{j}}^{\overline{l_{1}-j}}\otimes
 \alpha_{\overline{l_{2}}}(v_{\overline{j'}}^{\overline{l_{2}-j'}}).
\end{array}\]
\end{proof}

The next lemma can express the module $M(i,\overline{j})$ as the
tensor product of  $M(i,\overline{0})$ and $M(1,\overline{1})$'s,
which will be used to simplify calculations in the following
discussion. We call this result the \emph{elementary lemma}.

\begin{Lem}\label{lem3.3}{\rm (}{\bf Elementary Lemma}{\rm )}~
{\rm $M(i,\overline{j})\otimes M(1,\overline{j'})\cong
M(1,\overline{j'})\otimes M(i,\overline{j})\cong
M(i,\overline{j+j'})$ for all $1\leq i\leq d$.}
\end{Lem}

\begin{proof}
By Lemma \ref{lem3.1}, we have $({\bf dim} M(i,\overline{j})\otimes
M(1,\overline{j'}))_{\overline{l}}=\left\{\begin{array}{ll}
1, &\mbox{$j+j'\leq l\leq j+j'+i-1$} \\
0, &\mbox{otherwise}
\end{array}\right.$
and for any $j+j'\leq l\leq j+j'+i-1$, the basis of
$(M(i,\overline{j})\otimes M(1,\overline{j'}))_{\overline{l}}$ is
$v_{\overline{j}}^{\overline{l-j-j'}}\otimes
v_{\overline{j'}}^{\overline{0}}$.

Now it suffices to check the morphism coinciding with the morphism
of $M(i,\overline{j+j'})$. Actually, by Lemma \ref{lem3.2}, for
$j+j'\leq l\leq j+j'+i-1$, choose $l=l_{1}+l_{2}=(l-j')+j'$, then we
have

\[\begin{array}{ccl}
\alpha_{\overline{l}}\cdot(v_{\overline{j}}^{\overline{l-j'-j}}\otimes
v_{\overline{j'}}^{\overline{0}}) & = &
\alpha_{\overline{l-j'}}(v_{\overline{j}}^{\overline{l-j'-j}})\otimes
v_{\overline{j'}}^{\overline{0}}+v_{\overline{j}}^{\overline{l-j'-j}}\otimes \alpha_{\overline{j'}}(v_{\overline{j'}}^{\overline{0}})  \\
& = & \alpha_{\overline{l-j'}}(v_{\overline{j}}^{\overline{l-j'-j}})\otimes v_{\overline{j'}}^{\overline{0}}  \\
& = & \left\{\begin{array}{ll}
v_{\overline{j}}^{\overline{l-j'+1-j}}\otimes v_{\overline{j'}}^{\overline{0}}, &\mbox{$j+j'\leq l\leq j+j'+i-2$} \\
0, &\mbox{otherwise}
\end{array}\right.
\end{array}\]
therefore $M(i,\overline{j})\otimes M(1,\overline{j'})\cong
M(i,\overline{j+j'})$. The other one is similar.
\end{proof}

By the Elementary Lemma we can restrict our study to the form
$M(i,\overline{0}), 1\leq i\leq d$ and $M(1,\overline{1})$. Indeed,
for any $1\leq i,i'\leq d, 0\leq j,j'\leq n-1$, we have
\begin{equation}\label{fromelementary}
M(i,\overline{j})\otimes M(i',\overline{j'})\cong
M(i,\overline{0})\otimes M(i',\overline{0})\otimes
M(1,\overline{1})^{\otimes (j+j')},
\end{equation}
where $M(1,\overline{1})^{\otimes (j+j')}$ means $\underbrace
{M(1,\overline{1})\otimes M(1,\overline{1})\otimes\cdots
M(1,\overline{1})}_{(j+j')-\text{times}}$ for simplify; in the sequel,
the meaning of $M(i,\overline{j})^{\otimes k}$, in general, is the
same.

Additionally, if $(M(i,\overline{0})\otimes
M(i',\overline{0}))_{\overline{l}}\neq 0$, then the basis of
$M(i,\overline{0})\otimes M(i',\overline{0})$ is
$$\{v_{\overline{0}}^{\overline{k}}\otimes v_{\overline{0}}^{\overline{k'}}|k+k'=l,0\leq k \leq i-1,0\leq k'\leq
i'-1\}.$$ For simplify, from now on, we abbreviate
$v_{\overline{0}}^{\overline{k}}\otimes
v_{\overline{0}}^{\overline{k'}}$ as $v^{\overline{k}}\otimes
v^{\overline{k'}}$.

\section{{ Indecomposable modules associated with the Pascal triangle }}

In this section, we give a combinatorial way to construct the
indecomposable submodules $M$ of $M(i,\overline{0})\otimes
M(i',\overline{0})$ via the Pascal triangle.

{\bf Step $1$.} First we introduce some notations. We label the row of Pascal triangle as $0$-th
row, $1$-th row, $\cdots$ from the top to the bottom, label the symmetry axis of Pascal
triangle
as $0$-th column; $1$-th column, $2$-th column from this axis to the right; $(-1)$-th
column,
$(-2)$-th column from this axis to the left, and call the point in the $i$- row and $j$-th
column
as $(i,j)$ position, see Figure $3$.

\begin{figure}[h] \centering
  \includegraphics*[131,579][445,812]{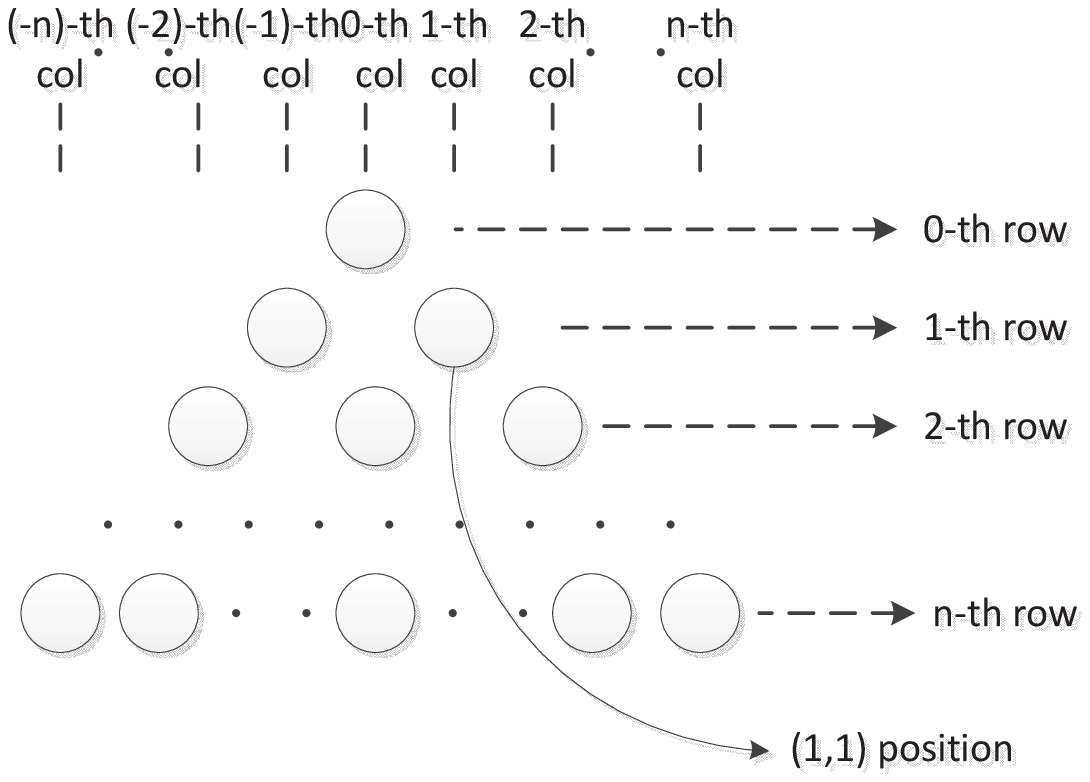}

 {\rm Figure 3 }
\end{figure}

{\bf Step $2$.} Next for the initial data $(i,i',l,\{u_{j}\}_{0\leq j\leq l})$, we will complete
the Pascal triangle by filling one number in each position. More precisely, for $0\leq
i,i',l\leq d-1$ and a sequence of integer coefficients
$U=\{u_{j}\}_{0\leq j\leq l}$, we
fill the Pascal triangle in the following way:

(i)~ For the entries in the $(i,j)$ position with $0\leq i<l$, let $\diamondsuit (i,j)=0$,
where
$\diamondsuit(i,j)$ denotes the number filling in the $(i,j)$ position.

(ii)~ For the entries in the $(l,j)$ position, let $\diamondsuit
(l,-l)=u_{0},\diamondsuit(l,-l+2)=u_{1},\cdots,\diamondsuit(l,l)=u_{l}$.

(iii)~ For the entries in the $(i,j)$ position with $i>l$, finish the Pascal triangle via
the
Pascal triangle rules, that is, if $i+j=0$, then let
$\diamondsuit(i,j)=\diamondsuit(i-1,j+1)$; If
$i=j$, then let $\diamondsuit(i,j)=\diamondsuit(i-1,j-1)$; Otherwise let
$\diamondsuit(i,j)=\diamondsuit(i-1,j-1)+\diamondsuit(i-1,j+1)$.

{\bf Step $3$.} At last we associate each vertex
$v_{\overline{l+k}}$ a vector space $V_{\overline{l+k}}$ in the
following way. It has basis
$\diamondsuit(l+k,-l-k)~v^{\overline{l+k}}\otimes
v^{\overline{0}}+\cdots+\diamondsuit(l+k,l+k)~v^{\overline{0}}\otimes
v^{\overline{l+k}}$. Here are some remarks.

(i)~ We have $k\geq 0$ and $v_{\overline{n+s}}=v_{\overline{s}},
s\geq 0$.

(ii)~ If $v^{\overline{s}}$ is in the left side of the
$\otimes$ symbol and $s\geq i$, then $v^{\overline{s}}=0$.

(iii)~ If $v^{\overline{t}}$ is in the right side of the
$\otimes$ symbol and $t\geq i'$, then $v^{\overline{t}}=0$.

Now we associate a $KZ_{n}/J^{d}$-module with a given data $(i,i',l,\{u_{j}\}_{0\leq j\leq l})$ as follows.

\begin{Def}\label{def5.1}
{\rm Given an initial data $(i,i',l,\{u_{j}\}_{0\leq j\leq l})$, define a $KZ_{n}/J^{d}$-module $M$ as follows:

(i) $M=\bigoplus\limits_{k\geq 0} V_{\overline{l+k}}$ as a vector
space;

(ii) For each vertex $v_{\overline{k}}$ and its associative vector
space $V_{\overline{k}}$, the module action in $M$ is defined as
$$\alpha_{\overline{k}}\cdot (v^{\overline{l_{1}}}\otimes
v^{\overline{k-l_{1}}})=\alpha_{\overline{l_{1}}}(v^{\overline{l_{1}}})\otimes
v^{\overline{k-l_{1}}}+v^{\overline{l_{1}}}\otimes
\alpha_{\overline{k-l_{1}}}(v^{\overline{k-l_{1}}}),$$ which is just
the simplified form of the result of Lemma \ref{lem3.2}.

From this definition, we denote this module by $M=M(i,i',l,\{u_{j}\}_{0\leq j\leq l})$.}
\end{Def}

The following proposition shows that $M$ is indeed an indecomposable
submodule of $M(i,\overline{0})\otimes M(i',\overline{0})$.

\begin{Prop}\label{prop5.2}
{\rm For any $0\leq i,i',l\leq d-1$ and a sequence of integer
coefficients $U=\{u_{j}\}_{0\leq j\leq l}$. If $l\leq {\rm min}\{i,i'\}$, then the module
$M=M(i,i',l,\{u_{j}\}_{0\leq j\leq l})$ constructed above is an
indecomposable submodule of $M(i,\overline{0})\otimes
M(i',\overline{0})$. }
\end{Prop}

\begin{proof}
 Firstly, since $l\leq {\rm min}\{i,i'\}$, thus $V_{\overline{l}}$ is a subspace of $(M(i,\overline{0})\otimes
M(i',\overline{0}))_{\overline{l}}$. Moreover, $M$ inherits the
module structure of $M(i,\overline{0})\otimes M(i',\overline{0})$ and $M$ is generated
by $w_0=u_{0}v^{\overline{l}}\otimes
v^{\overline{0}}+u_{1}v^{\overline{l-1}}\otimes
v^{\overline{1}}+\cdots+u_{l}v^{\overline{0}}\otimes
v^{\overline{l}}$,
hence $M$ is a submodule of $M(i,\overline{0})\otimes
M(i',\overline{0})$.

Moreover, the indecomposability of $M$ can be obtained as a consequence of the
actions of the morphisms
$\alpha_{\overline{l}},\alpha_{\overline{l+1}},\cdots$ on the vector
spaces $V_{\overline{l}},V_{\overline{l+1}},\cdots$, respectively,
according to the construction of $M=\bigoplus\limits_{k\geq 0}
V_{\overline{l+k}}$. In fact, by definition of $M$, it is generated
by the element $w_0$, which is a basis of $V_{\overline{l}}$. If $M$ is decomposable, then let $M=M'\oplus M''$ as module, then $w_0\in M'$ or $w_0\in M''$.
Assume $w_0\in M'$. Thus $M'\cap V_{\overline{l}}\neq \{0\}$. But
dim$V_{\overline{l}}=1$, we get $V_{\overline{l}}\subseteq (M')_{\overline{l}}$. So, $M=\bigoplus\limits_{k\geq 0} V_{\overline{l+k}}$ is a
submodule of $M'$ since $M$ is generated by $w_{0}$. It means that $M=M'$
and $M''=0$, which implies that $M$ is indecomposable.
\end{proof}

By Proposition \ref{prop5.2}, the coefficient of
$v^{\overline{i}}\otimes v^{\overline{j}}$ is stored in the
$(i+j,j-i)$ position. Therefore from now on, we write this
phenomenon as $c(i,j)=\diamondsuit(i+j,j-i)$ for short.

We end this section with the following example for constructing module $M=M(i,i',l,\{u_{j}\}_{0\leq j\leq l})$.

\begin{Ex}\label{ex5.3}
{\rm Let $i=i'=2, l=0, u_{0}=1, n=6, d=p=5$, then
$M(2,2,0,\{1\})=\oplus_{k=0}^5 V_{\overline{k}}$ is an
indecomposable submodule of $M(2,\overline{0})\otimes M(2,\overline{0})$ by Proposition
\ref{prop5.2}.
 Through the Pascal triangle we have the following processes.

(i)~ $V_{\overline{0}}={\rm span}\{v^{\overline{0}}\otimes
v^{\overline{0}}\}\cong K$;

 (ii)~ $V_{\overline{1}}={\rm span}\{v^{\overline{1}}\otimes v^{\overline{0}}+v^{\overline{0}}\otimes v^{\overline{1}}\}\cong K$, since $\alpha_{\overline{0}}\cdot (v^{\overline{0}}\otimes v^{\overline{0}})=v^{\overline{1}}\otimes v^{\overline{0}}+v^{\overline{0}}\otimes v^{\overline{1}}$;

 (iii)~ $V_{\overline{2}}={\rm span}\{2 v^{\overline{1}}\otimes v^{\overline{1}}\}\cong K$, since $\alpha_{\overline{1}}\cdot
 (v^{\overline{1}}\otimes v^{\overline{0}}+v^{\overline{0}}\otimes v^{\overline{1}})=v^{\overline{2}}\otimes v^{\overline{0}}+2 v^{\overline{1}}\otimes v^{\overline{1}}+v^{\overline{0}}\otimes v^{\overline{2}}=2 v^{\overline{1}}\otimes v^{\overline{1}}$ with $v^{\overline{2}}=\alpha_{\overline{1}}\alpha_{\overline{0}}=0$ in
 $M(2,\overline{0})$;

 (iv)~ $V_{\overline{3}}=V_{\overline{4}}=V_{\overline{5}}=0$, since $\alpha_{\overline{2}}\cdot (2 v^{\overline{1}}\otimes v^{\overline{1}})=2 v^{\overline{2}}\otimes v^{\overline{1}}+2 v^{\overline{1}}\otimes v^{\overline{2}}=0$.  \\
Thus, $M(2,2,0,\{1\})={\rm span}\{v^{\overline{0}}\otimes
v^{\overline{0}}, v^{\overline{1}}\otimes
v^{\overline{0}}+v^{\overline{0}}\otimes v^{\overline{1}}, 2
v^{\overline{1}}\otimes v^{\overline{1}}\}$.

In Figure 4, the right figure expresses $M(2,2,0,\{1\})$, in which
the coefficients occurring in the basis of $M(2,2,0,\{1\})$ are
surrounded by the dotted line; the left figure expresses
$M(3,\overline{0})$. The right figure can be induced from the left
figure under the action of arrows; conversely, the left figure is
directly obtained from the right figure. Hence, from Figure 4, we
have shown that $M(2,2,0,\{1\})\cong M(3,\overline{0})$, which is an indecomposable submodule of $M(2,\overline{0})\otimes M(2,\overline{0})$.  }

\begin{figure}[h] \centering
  \includegraphics*[126,341][465,501]{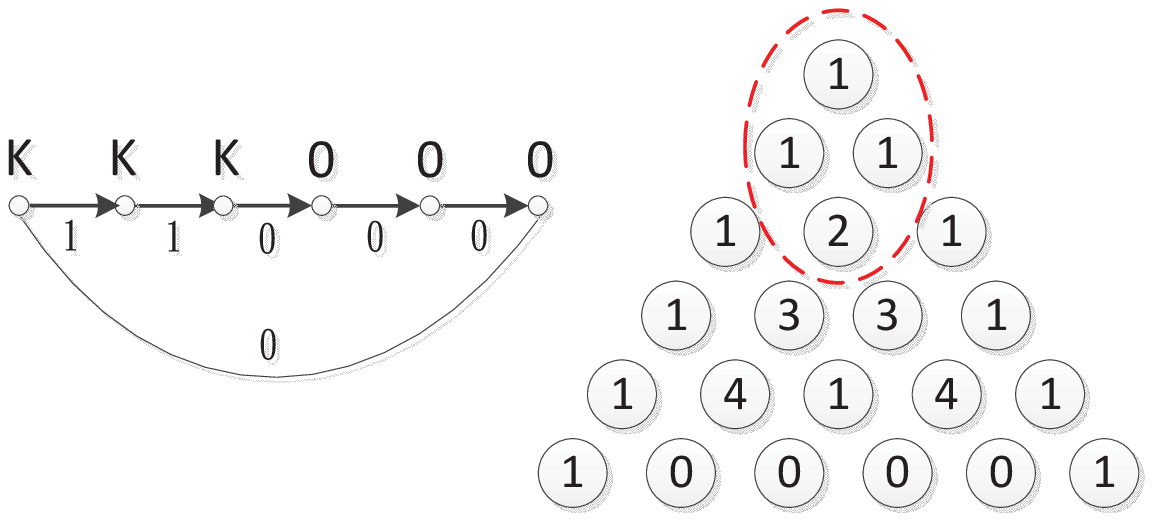}

 {\rm Figure 4 }
\end{figure}
\end{Ex}

\bigskip
\section{{ The generators of the representation ring }}

This section is devoted to calculating the generators of representation ring $r(KZ_{n}/J^{d})$.

We first give the following Lemma which is similar to the Lemma 3.8 in \cite{[7]}, here we
give another verification by the notation introduced in Section $5$.

\begin{Lem}\label{lem6.1}
{\rm (i)~ For all $1\leq i\leq d, 0\leq j\leq n-1,$~~
$$M(i,\overline{j})\otimes M(1,\overline{0})\cong M(1,\overline{0})\otimes M(i,\overline{j}) \cong M(i,\overline{j}),
\;\;\;\; M(1,\overline{1})^{\otimes n}\cong M(1,\overline{0}).$$

(ii)~ $M(2,\overline{0})\otimes M(t,\overline{0})\cong
M(t+1,\overline{0})\oplus M(t-1,\overline{1})$\;\;\; for all $t\geq
2, p \nmid t$.

(iii)~ $M(2,\overline{0})\otimes M(t,\overline{0})\cong
M(t,\overline{0})\oplus M(t,\overline{1})$\;\;\; for all $t>0, p|t$.
\\ }
\end{Lem}

\begin{proof}
(i)~ Follows directly from Lemma \ref{lem3.3}.

(ii)~ For $t\geq 2, p\nmid t$, by Proposition \ref{prop5.2} we have
two submodules $M_{1}=M(2,t,0,\{1\})$ and $M_{2}=M(2,t,1,\{1-t,1\})$
of $M(2,\overline{0})\otimes M(t,\overline{0})$. The coefficients of
the bases of $(M_1)_{\overline{l}}, (M_2)_{\overline{l}}$, $0\leq l\leq n-1$ can be seen in the following table.

\setlength{\tabcolsep}{7pt}
\renewcommand{\arraystretch}{1.3}
\begin{center}
\begin{tabular}{|c|c|c|c|c|c|c|c|c|c|c|}
\hline
\multicolumn{1}{|c|}{Rows in Pascal triangle} & \multicolumn{1}{|c|}{$0$} & \multicolumn{1}{|c|}{$1$} &
\multicolumn{1}{|c|}{$\cdots$} & \multicolumn{1}{|c|}{$l$} & \multicolumn{1}{|c|}{$\cdots$}
&
\multicolumn{1}{|c|}{$t-1$} & \multicolumn{1}{|c|}{$t$} & \multicolumn{1}{|c|}{$t+1$} &
\multicolumn{1}{|c|}{$\cdots$} & \multicolumn{1}{|c|}{$n-1$}    \\
\hline
$M_{1}$ & $(0,1)$ & $(1,1)$ & $\cdots$ & $(l,1)$ & $\cdots$ & $(t-1,1)$ & $(t,0)$ & $(0,0)$
&
$\cdots$ & $(0,0)$ \\
\hline
$M_{2}$ & $(0,0)$ & $(1-t,1)$ & $\cdots$ & $(l-t,1)$ & $\cdots$ & $(-1,1)$ & $(0,0)$ &
$(0,0)$ &
$\cdots$ & $(0,0)$  \\
\hline
\end{tabular}
\end{center}
\bigskip

where the coefficients pair $(c_{1},c_{2})$ in the $s$-position means that
$c_{1}=c(1,s-1),~ c_{2}=c(0,s)$, which are stored in the upper-right two diagonals of the Pascal triangle.

Thus, for $p\nmid t$, we have $M_{1}\cong M(t+1,\overline{0})$,
$M_{2}\cong M(t-1,\overline{1})$, ${\rm
dim}(M_{1})_{\overline{l}}+{\rm dim}(M_{2})_{\overline{l}}={\rm dim}
(M(2,\overline{0})\otimes M(t,\overline{0}))_{\overline{l}}$ and
$(M_{1})_{\overline{l}}\cap (M_{2})_{\overline{l}}=\{0\}$ for any
$0\leq l\leq n-1$. Therefore,
$$M(2,\overline{0})\otimes M(t,\overline{0})\cong M_{1} \oplus M_{2} \cong M(t+1,\overline{0})\oplus
M(t-1,\overline{1}).$$

(iii) Similarly with (ii), we just remark that
$M(t,\overline{0})\cong M(2,t,0,\{1\})$, $M(t,\overline{1})\cong
M(2,t,1,\{1,0\})$. The basis of $M(t,\overline{0})_{\overline{l}}$
is $v^{\overline{0}}\otimes
v^{\overline{l}}+l~v^{\overline{1}}\otimes v^{\overline{l-1}}$ and
the basis of $M(t,\overline{1})_{\overline{l}}$ is
$v^{\overline{1}}\otimes v^{\overline{l-1}}$ for $1\leq l\leq t$.
\end{proof}

Actually, the coefficients of the basis can be read from the Pascal triangles directly, see
Figure $5$.

\begin{figure}[ht] \centering
  \includegraphics*[23,245][494,522]{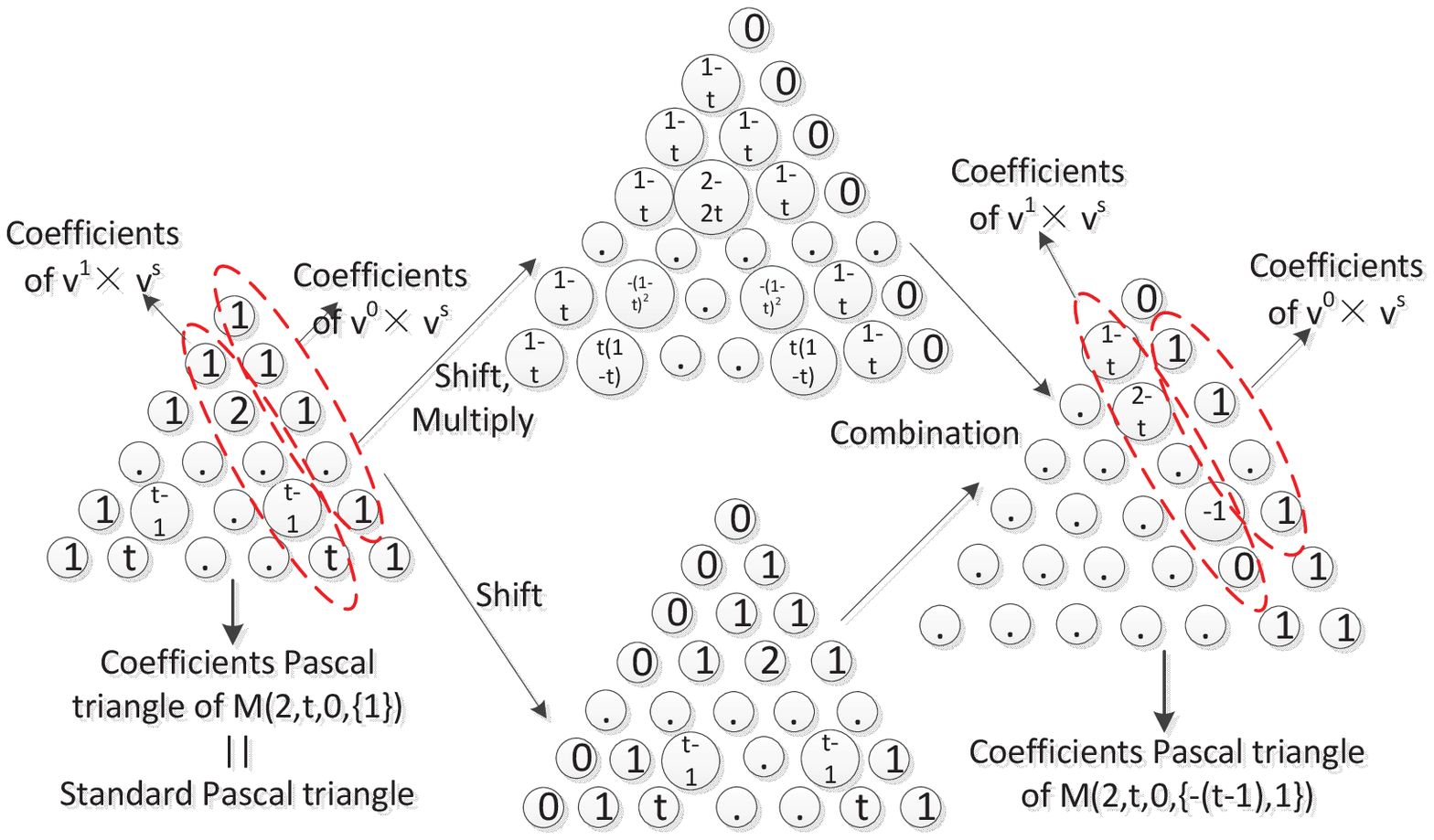}

 {\rm Figure 5 }
\end{figure}

In order to determine the generators of the Green ring $r(KZ_{n}/J^{d})$, we need more
observations.

\begin{Lem}\label{lem6.2}
{\rm For any $1\leq u\leq d$, the isomorphism class
$[M(u,\overline{0})]$ of the indecomposable module
$M(u,\overline{0})$ can be expressed as a polynomial of
$[M(1,\overline{1})], [M(2,\overline{0})],
[M(p^{l}+1,\overline{0})], \forall 1\leq l\leq m-1$ in
$r(KZ_{n}/J^{d})$.}
\end{Lem}
\begin{proof}
In the case that $u-1$ is a power of $p$, the conclusion follows at once.

Otherwise, in the case that $u-1$ is not a power of $p$, we prove this conclusion by using of the induction on $u$.

When $u=1$, the conclusion is trivial.

When $u>1$ and $u-1$ is not a power of $p$, assuming that for any
$1\leq v<u$, $M(v,\overline{0})$ can be expressed as a polynomial of
$[M(1,\overline{1})], [M(2,\overline{0})],
[M(p^{l}+1,\overline{0})], \forall 1\leq l\leq m-1$ in
$r(KZ_{n}/J^{d})$. By the fact {\rm
gcd}$(\binom{u-1}{1},\binom{u-1}{2},\cdots,\binom{u-1}{u-2})=1$ in
\cite{[16]}, there exists $w, 1\leq w\leq u-2$ such that
$\binom{u-1}{w}\neq 0$ in $K$ with char$K=p$.

Now, we will prove the result on $M(u,\overline{0})$ by the induction
hypotheses. For this aim, consider the decomposition of
$M(u-w,\overline{0})\otimes M(w+1,\overline{0})$ into
indecomposables.

 Firstly,  since $\binom{u-1}{w}\neq 0$ in $K$, $M(u-w,w+1,0,\{1\})$ is a submodule of $M(u-w,\overline{0})\otimes M(w+1,\overline{0})$, generated by $v^{\overline{0}}\otimes
v^{\overline{0}}$ and the corresponding figure in the Pascal
triangle is the rectangle with points at $(0,0)$, $(u-w-1,0)$,
$(0,w)$, $(u-w-1,w)$ positions. Then it is clear that
$M(u,\overline{0})\cong M(u-w,w+1,0,\{1\})$.

Given the decomposition of $M(u-w,\overline{0})\otimes M(w+1,\overline{0})$ into the direct sum of indecomposables as that
$M(u-w,\overline{0})\otimes M(w+1,\overline{0})=M_{1}\oplus
M_{2}\cdots \oplus M_{t}$, since ${\rm dim}
(M(u-w,\overline{0})\otimes M(w+1,\overline{0}))_{\overline{0}}=1$,
there exists a unique $s$ with $1\leq s\leq t$ such that ${\rm dim}
(M_{s})_{\overline{0}}=1$. Without loss of generality, say $s=1$, then it follows $v^{\overline{0}}\otimes v^{\overline{0}}\in M_1$. Then
$M(u,\overline{0})$ becomes a submodule of $M_{1}$ since
$M(u,\overline{0})$ is generated by $v^{\overline{0}}\otimes
v^{\overline{0}}$.

 Because $M_{1}$ is a submodule of $M(u-w,\overline{0})\otimes M(w+1,\overline{0})$, we have ${\rm dim}
(M_{1})_{\overline{l}}=0$ for $u\leq l\leq n-1$. Thus, ${\rm dim}
M_{1}\leq (u-1)-0+1=u={\rm dim} M(u,\overline{0})$.

Therefore $M_{1}=M(u,\overline{0})$ and then we have
$M(u-w,\overline{0})\otimes
M(w+1,\overline{0})=M(u,\overline{0})\oplus M_{2}\cdots \oplus
M_{t}.$ From this, in $r(KZ_{n}/J^{d})$, we get that
\begin{equation}\label{usefulone}
[M(u,\overline{0})]=[M(u-w,\overline{0})][M(w+1,\overline{0})]-[M_{2}]-\cdots
-[M_{t}].
\end{equation}

For any $2\leq s\leq t$, there are $1\leq i_{s}, j_{s}\leq u-2$ such
that $M_{s}\cong M(i_{s},\overline{j_{s}})$ and by the Elementary
Lemma, $M(i_{s},\overline{j_{s}})=M(i_{s},\overline{0})\otimes
M(1,\overline{1})^{\otimes j_{s}}$. Thus
$[M_s]=[M(i_{s},\overline{j_{s}})]=[M(i_{s},\overline{0})][M(1,\overline{1})]^{j_{s}}$.

Hence, by the induction hypotheses, all of
$[M(u-w,\overline{0})], [M(w+1,\overline{0})], [M_{2}], \cdots,
[M_{t}]$ can be expressed as a polynomial of $[M(1,\overline{1})],
[M(2,\overline{0})], [M(p^{l}+1,\overline{0})], \forall 1\leq
l\leq m-1$ in $r(KZ_{n}/J^{d})$.
 By (\ref{usefulone}), so does $[M(u,\overline{0})]$, too.
\end{proof}

By Lemma \ref{lem3.3}, we have
$[M(u,\overline{w})]=[M(u,\overline{0})][M(1,\overline{1})]^{w}, \forall
1\leq u\leq d, 0\leq w\leq n-1$. By this fact and Lemma
\ref{lem6.2}, we obtain the following.

\begin{Thm}\label{thm6.3}
{\rm Let {\rm char}$K=p$ and $n\geq d=p^m$. Then the representation
ring $r(KZ_{n}/J^{d})$ of the Nakayama truncated algebra
$KZ_{n}/J^{d}$ is generated by
$[M(1,\overline{1})],[M(2,\overline{0})]$ and
$[M(p^{l}+1,\overline{0})],1\leq l\leq m-1$.}
\end{Thm}

The following corollary is more precise.

\begin{Cor}\label{lem6.4}
{\rm For any $1\leq u\leq p^{l}, 1\leq l\leq m, 0\leq r\leq n-1$, then the element $[M(u,\overline{r})]$ in
$r(KZ_{n}/J^{d})$ can be expressed
as a polynomial of $$[M(1,\overline{1})],\; [M(2,\overline{0})],\;
[M(p+1,\overline{0})],\; [M(p^{2}+1,\overline{0})],\; \cdots,\;
[M(p^{l-1}+1,\overline{0})].$$ }
\end{Cor}

\begin{proof}
We may use the induction on $l$. When $l=1$, by Lemma \ref{lem3.3}, we have
$$[M(i,\overline{j})]=[M(i,\overline{0})][M(1,\overline{j})]=[M(i,\overline{0})][M(1,\overline{1})]^{j}$$ and by Lemma \ref{lem6.1}, recursively, we have

$[M(2,\overline{0})][M(2,\overline{0})]=[M(3,\overline{0})]+[M(1,\overline{1})]$,
$[M(2,\overline{0})][M(3,\overline{0})]=[M(4,\overline{0})]+[M(2,\overline{1})]$,
$\cdots$,
\\$[M(2,\overline{0})][M(p-1,\overline{0})]=[M(p,\overline{0})]+[M(p-2,\overline{1})]$.

It follows that for any $1\leq u\leq p, 0\leq r\leq n-1$,
$[M(u,\overline{r})]$ can be expressed as a polynomial of
$[M(1,\overline{1})], [M(2,\overline{0})]$.

 Assume the result is true for $l-1$. Now consider it for $l$. Let $1\leq u\leq p^{l}, 0\leq r\leq n-1$.

 If $1\leq u\leq p^{l-1}$, then the result follows by the induction hypothesis. So, we only need to consider the case
 $p^{l-1}+1\leq u\leq p^{l}$ by using induction again on $u$.

Firstly, if $u=p^{l-1}+1$, then the result is trivial since we have
$[M(u,\overline{r})]=[M(p^{l-1}+1,\overline{0})][M(1,\overline{1})]^{r}$
by Lemma \ref{lem3.3}. Now, suppose the results holds for all $u'<u$
and then consider $[M(u,\overline{0})]$.

 In this case $p^{l-1}+2\leq u\leq p^{l}$,  $u-1$ is not a power of $p$. According to the proof of Lemma \ref{lem6.2}, there exists $1\leq w\leq u-2$ such that $\binom{u-1}{w}\neq 0$ in $K$ and the following holds:
\begin{equation}\label{usefultwo}
[M(u,\overline{0})]=[M(u-w,\overline{0})][M(w+1,\overline{0})]-[M_{2}]-\cdots
-[M_{t}].
\end{equation}
where for any $2\leq s\leq t$, $M_{s}\cong
M(i_{s},\overline{j_{s}})$ for some $i_{s}, j_{s}$ with $1\leq
i_{s}, j_{s}\leq u-2$.

 Finally, the result follows by the induction hypotheses applying for the terms $[M(u-w,\overline{0})],[M(w+1,\overline{0})],[M_{2}],\cdots,[M_{t}]$.
\end{proof}

From this result, there is a unique ring epimorphism
$\phi:\;\mathcal{Z}[y,z,w_{1},\cdots,w_{m-1}]\longrightarrow
r(KZ_{n}/J^{d})$ with
$\phi(1)=[M(1,\overline{0})],\phi(y)=[M(1,\overline{1})],\phi(z)=[M(2,\overline{0})]$
and $\phi(w_{l})=[M(p^{l}+1,\overline{0})],\; \text{for}\; 1\leq
l\leq m-1$.

In the next section, we start to show the relations among the
generators $[M(1,\overline{1})], [M(2,\overline{0})],
[M(p^{l}+1,\overline{0})], \forall 1\leq l\leq m-1$ of the
representation ring $r(KZ_{n}/J^{d})$ and their corresponding
preimages $y,z,w_{1}, \cdots,w_{m-1}$.

\bigskip
\section{{ Polynomial characterization of $r(KZ_{n}/J^{d})$ and application to isomorphism problem}}

This section is devoted to discussion on the relations among the generators of the representation ring $r(KZn/J^{d})$. We will need the following fact.

\begin{Lem}\label{lem7.2}
{\rm Let $R$ be an unital ring and $N_1$, $N_2$ be two unital $R$-modules. If $N_1\oplus
N_2$ has a simple submodule $S$, then either $N_1$ or $N_2$ has a simple
submodule isomorphic to $S$.}
\end{Lem}

\begin{proof}
Since $S$ is a simple module, then $S=Rb$ for some $b\neq 0$ in $N_1\oplus N_2$. Let $b=b_1+b_2$ for $b_1\in N_1$, $b_2\in
N_2$.  Without loss of generality, say $b_1\neq 0$. Then the annihilator of $b$, that is, ${\rm Ann}(b)=\{r\in R~ |~ rb=0\}$ is
a maximal left ideal of $R$, since $R/{\rm Ann}(b)\cong Rb$ is simple as $R$-module.

However, ${\rm Ann}(b_1)$ is a proper left ideal of $R$ and  ${\rm Ann}(b_1)\supseteq {\rm Ann}(b)$. Therefore, ${\rm Ann}(b_1)={\rm Ann}(b)$ due to the maximum of  ${\rm Ann}(b)$. Then as a submodule of $N_1$, we have $Rb_1\cong R/{\rm Ann}(b_1)=R/{\rm Ann}(b)\cong S$.
\end{proof}

Now, we begin with the following crucial statement.

\begin{Lem}\label{thm7.1}
{\rm {$M(p^{l}+1,\overline{0})\otimes M(kp^{l}+1,\overline{0}) \cong
\left\{\begin{array}{ll}
W_{1}, &\mbox{if~ $k\equiv -1({\rm mod}~p)$} \\
W_{2}, &\mbox{if~ $k\equiv 0({\rm mod}~p)$} \\
W_{3},  &\mbox{otherwise} \\
\end{array}\right.$}
where\\ $W_{1}=M(kp^{l}+p^{l},\overline{0})\oplus
M(kp^{l}+p^{l},\overline{1})\oplus
(\oplus_{j=2}^{p^{l}-1}M(kp^{l},\overline{j}))\oplus M(kp^{l}-p^{l}+1,\overline{p^{l}})$,\\
$W_{2}=M(kp^{l}+p^{l}+1,\overline{0})\oplus(\oplus_{j=1}^{p^{l}}M(kp^{l},\overline{j}))$,\\
$W_{3}=M(kp^{l}+p^{l}+1,\overline{0})\oplus
M(kp^{l}+p^{l}-1,\overline{1})\oplus(\oplus_{j=2}^{p^{l}-1}
M(kp^{l},\overline{j}))\oplus M(kp^{l}-p^{l}+1,\overline{p^{l}}).$ }
\end{Lem}

\begin{proof}
Here we only give the precise proof for the $W_{3}$ case with $l=1, p>2$. For the remaining cases, the proofs are similar.

We first construct a sequence of submodules $\{M_{i}\}_{0\leq i\leq
p}$ of $M(p+1,\overline{0})\otimes M(kp+1,\overline{0})$ via the
Pascal triangle, and then prove that
\begin{equation}\label{moduledef1}
M_{0}\cong M(kp+p+1,\overline{0}), M_{1}\cong
M(kp+p-1,\overline{1}),M_{p}\cong M(kp-p+1,\overline{p}),
\end{equation}
\begin{equation}\label{moduledef2}
M_{2}\cong M(kp,\overline{2}), M_{3}\cong M(kp,\overline{3}),\cdots,
M_{p-1}\cong M(kp,\overline{p-1}).
\end{equation}
Lastly show that $M(p+1,\overline{0})\otimes M(kp+1,\overline{0})\cong
\bigoplus_{i=0}^{p}M_{i}$.

{\bf Step $1$.} We define the submodules $\{M_{i}\}_{0\leq i\leq p}$ as follows:
$$M_{l}=\left\{\begin{array}{ll}
M(p+1,kp+1,l,\{0,\cdots,0,1,0,\cdots,0\}), &\mbox{if $l$ is even;} \\
M(p+1,kp+1,l,\{0,\cdots,0,1,-k,0,\cdots,0\}), &\mbox{if $l$ is odd}
\end{array}\right.$$
where the numbers of $0$ on the left-side and the right-side of $1$ are the same, so are on the left-side and the right-side of $\{1,-k\}$, too.

{\bf Step $2$.} Claim that $M_{0}\cong M(kp+p+1,\overline{0})$ and for
even $l=2,4,\cdots,p-1$, $M_{l}\cong M(kp,\overline{l})$. For odd
$l$, the proof is similar.

Indeed, the Pascal triangle associated with $M_{0}$ is given on the left of the following Figure 6, where the coefficients are surrounded by the red
dotted rectangle.
\begin{figure}[ht] \centering
  \includegraphics*[3,422][531,717]{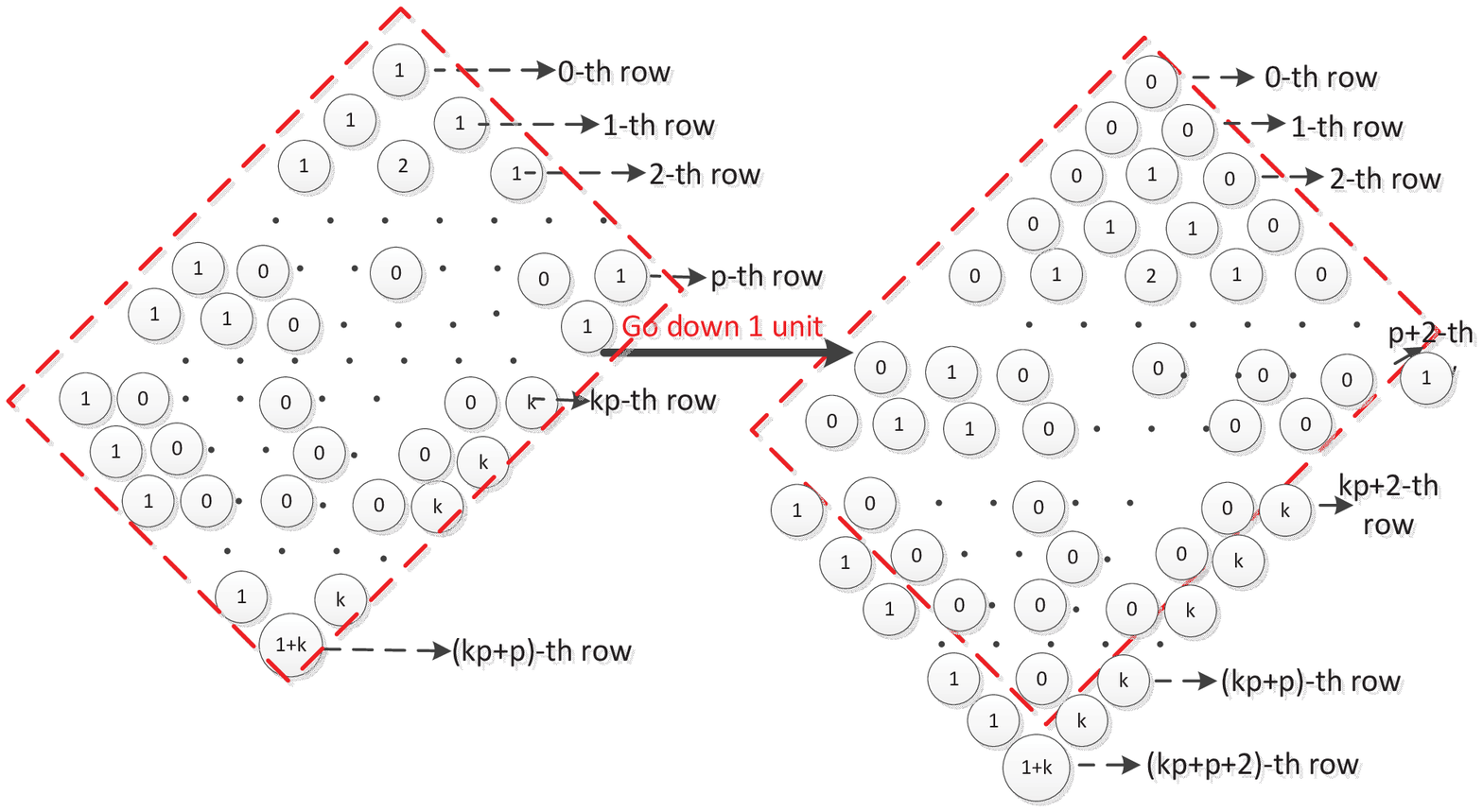}

 {\rm Figure 6 }
\end{figure}

For example, since $\binom{p}{0}=\binom{p}{p}=1$ and
$\binom{p}{i}=0$ in $K$ for $1\leq i\leq p-1$, we see that the
$p$-row of Pascal triangle associated to $M_{0}$ is
$\{1,0,\cdots,0,1\}$. By the definition of $M(kp+p+1,\overline{0})$,
it follows $M_{0}\cong M(kp+p+1,\overline{0})$ at once.

Furthermore, operating $M_{0}$'s Pascal triangle by going down $1$
unit, we obtain the Pascal triangle associated with $M_{2}$, that
is, that on the right-side of Figure 6, where the coefficients
are again surrounded by the red dotted rectangle. It is clear that
$M_{2}\cong M(kp,2)$ by the definition of $M(kp,\overline{2})$.

Continue this process, through operating $M_{0}$'s Pascal triangle
by going down $2,3,\cdots,\frac{p-1}{2}$ units, it follows that
$M_{l}\cong M(kp,\overline{l})$ for even $l=4,\cdots,p-1$.

{\bf Step $3$.} It remains to show that $M(p+1,\overline{0})\otimes
M(kp+1,\overline{0})\cong \bigoplus_{l=0}^{p}M_{l}.$

By Proposition \ref{prop5.2}, all $M_{l}$, as well as
$\sum_{l=0}^{p} M_{l}$, are submodules of
$M(p+1,\overline{0})\otimes M(kp+1,\overline{0})$.
 Next, we claim that $\sum_{l=0}^{p} M_{l}$ is indeed a direct sum by induction.

 For any $i$ with $0\leq i\leq n-1$, we denote by  $\beta_{l,i}$ the basis of the vector space $\{(M_{l})_{\overline{i}}\}$  when $(M_{l})_{\overline{i}}\neq 0$ and denote by ${\rm P}_{l}$ the Pascal triangle corresponding to the module $M_{l}$.

It is easy to see that $T_0=S_{\overline{kp+p+1}}$,
$T_1=S_{\overline{kp+p-1}}$, $T_i=S_{\overline{kp+i-1}}, \forall
2\leq i\leq p-1$, $T_p=S_{\overline{kp}}$ is respectively  the
unique simple submodule of $M_{0}$, $M_{1}$, $M_{i}, \forall 2\leq
i\leq p-1$, $M_{p}$.

 When $p=1$, we have $M_0\cap M_1=\{0\}$ since $T_0 \ncong
T_1$. Hence $M_0+M_1$ is a direct sum.

Assume that for $p-1$ the result holds, that is, $\sum_{j=0}^{p-1} M_{l_j}$ is a
direct sum for any subset $\{l_0,l_1,\cdots,l_{p-1}\}$ of $\{0,1,\cdots,p\}$. Now consider the case of $p$ for the sum $\sum_{l=0}^{p} M_{l}$.

By Lemma \ref{lem7.2} and the induction assumption, the direct sum $\sum_{j=0}^{p-1} M_{l_j}=\bigoplus_{j=0}^{p-1} M_{l_j}$ has just all $p$ simple submodules  $T_{l_0}$, $\cdots$, $T_{l_{p-1}}$, all of which are not isomorphic to $T_{l_p}$, the unique simple submodule of $M_{l_p}$  for $l_p$.
 Thus, $\bigoplus_{j=0}^{p-1} M_{l_j}$ has no any
simple submodule isomorphic to $T_{l_p}$. It follows that for any
$l_p\in\{0,1,\cdots,p\}$,
$$(\bigoplus \limits_{j=0}^{p-1} M_{l_j})\cap M_{l_p}=\{0\}.$$
  Therefore, $\sum\limits_{l=0}^{p}M_{l}=\sum\limits_{j=0}^{p}M_{l_j}=\bigoplus_{j=0}^{p}M_{l_j}$ is a direct sum.

 Finally, let us show that the dimension vectors of
$M(p+1,\overline{0})\otimes M(kp+1,\overline{0})$ and
$\bigoplus_{l=0}^{p}M_{l}$ are the same. It follows from the
equality below, that is, for each $i$ with $0\leq i\leq p$,
\begin{equation}\label{compare}
 {\rm dim}(M(p+1,\overline{0})\otimes M(kp+1,\overline{0}))_{\overline{i}}=\left\{
\begin{array}{ll}
i+1,  &\mbox{if ~$0\leq i\leq p-1;$}\\
p+1,  &\mbox{if ~$p\leq i\leq kp;$}\\
kp+p+1-i,  &\mbox{if ~$kp+1\leq i\leq kp+p;$}\\
0,  &\mbox{otherwise.}\\
\end{array}
\right. =\sum_{l=0}^{p} {\rm dim}(M_{l})_{\overline{i}}.
\end{equation}

In fact, the basis of $M(p+1,\overline{0})\otimes
M(kp+1,\overline{0})$ is $\{v^{\overline{s}}\otimes
v^{\overline{t}}|0\leq s\leq p, 0\leq t\leq kp\}$ and for each case
on $i$, ${\rm dim}(M(p+1,\overline{0})\otimes
M(kp+1,\overline{0}))_{\overline{i}}$ can be easily calculated to
obtain the value in (\ref{compare}) by the partition of the
rectangle $(0,0),(0,p),(kp,0),(kp,p)$ in the $x$-$y$-coordinate
system via the diagonals $x+y=l$. On the other hand, ${\rm dim}
(\sum_{l=0}^{p}M_{l})_{\overline{i}}=\sum_{l=0}^{p} {\rm
dim}(M_{l})_{\overline{i}}$, which equals to the middle value in
(\ref{compare}) for each case on $i$ according to the isomorphisms
$$M_{0}\cong M(kp+p+1,\overline{0}),~~ M_{1}\cong M(kp+p-1,\overline{1}),~~ M_{p}\cong
M(kp-p+1,\overline{p})$$
$$M_{2}\cong M(kp,\overline{2}),~~ M_{3}\cong M(kp,\overline{3}),~~\cdots,~~ M_{p-1}\cong M(kp,\overline{p-1}).$$
\end{proof}

\begin{Rem}
{\rm By Lemma \ref{thm7.1}, we can write down some relations about
the generators
$$[M(1,\overline{1})],[M(2,\overline{0})],[M(p^{l}+1,\overline{0})], 1\leq l\leq m-1$$
of the representation ring $r(KZ_{n}/J^{d})$. For example, choose
$n=d=27=3^{3}, m=1, {\rm char}K=p=3, l=1, k=3$, we obtain that
\begin{equation}\label{newrelation}
M(4,\overline{0})\otimes M(10,\overline{0})\cong
M(13,\overline{0})\oplus M(9,\overline{1})\oplus
M(9,\overline{2})\oplus M(9,\overline{3}).
\end{equation}

However, in the Green ring of the Taft algebra $KZ_{27}/J^{27}$ calculated in \cite{[7]}, the relation was that
\begin{equation}\label{oldrelation}
M(4,\overline{0})\otimes M(10,\overline{0})\cong
M(13,\overline{0})\oplus M(11,\overline{1})\oplus
M(9,\overline{2})\oplus M(7,\overline{3})
\end{equation}
 by the classical Clebsch-Gordan-like formula in \cite{[8]}.

 Note that the relations (\ref{newrelation}) and (\ref{oldrelation}) is various. The reason is that although the Taft algebras and the generalized Taft algebras in \cite{[7],[17]} are as the special cases of the Nakayama truncated algebras we are discussing, their coalgebra structures, as well as the Green rings, are different with that of the Nakayama truncated algebras. }
\end{Rem}

Now we can realize the representation ring $r(KZ_{n}/J^{d})$ as the quotients of polynomial ring by using Lemma \ref{thm7.1}. As preparation, we first give
the formula of a linear recurrent sequence. Given two initial values $a_{1},a_{2}$ and the recursion relation $$a_{n}=za_{n-1}-ya_{n-2},
~ n\geq 2$$ with fixed $y,z$ satisfying $z^{2}-4y\neq 0$, by the well-known result on linear recurrence sequences, we can write

\begin{equation}\label{lrssolution}
a_{n}=Ax_{1}^{n-1}+Bx_{2}^{n-1},
\end{equation}
where $x_{1}=\frac{z+\sqrt{\Delta}}{2}, x_{2}=\frac{z-\sqrt{\Delta}}{2}, \Delta=z^{2}-4y$ and
$A=\frac{2a_{2}-(z-\sqrt{\Delta})a_{1}}{2\sqrt{\Delta}},
B=\frac{(z+\sqrt{\Delta})a_{1}-2a_{2}}{2\sqrt{\Delta}}.$

Substituting $A,B,x_1,x_2$ into the formula (\ref{lrssolution}), we obtain that
\begin{equation}\label{lrs2}
 a_{n}=\frac{2a_{2}-(z-\sqrt{\Delta})a_{1}}{2\sqrt{\Delta}}(\frac{z+\sqrt{\Delta}}{2})^{n-1}+\frac{(z+\sqrt{\Delta})a_{1}-2a_{2}}{2\sqrt{\Delta}}(\frac{z-\sqrt{\Delta}}{2})^{n-1}.
\end{equation}
Write the right of (\ref{lrs2}) as a polynomial $g$ of $a_1,a_2$,
we denote $a_{n}$ as $g(n,a_{1},a_{2})$.

For the following discussion, we first define the order
\begin{equation}\label{order}
y<z<w_{1}<w_{2}<\cdots<w_{m-1}
\end{equation}
 in the polynomial ring
$\mathds{Z}[y,z,w_{1},\cdots,w_{m-1}]$ and then consider the order of the monomials
according to the dictionary order.

For a monomial $az_1^{i_1}\cdots z_n^{i_n}$ of a polynomial $f(z_1,\cdots,z_n)$, denote by $I(ax_1^{i_1}\cdots x_n^{i_n})=(i_1,\cdots,i_n)$, which is called the index series of the monomial $az_1^{i_1}\cdots z_n^{i_n}$.
Arranging its monomials in a descending order, the leading term of the
polynomial $f$ is the largest monomial, denoted as $lt(f)$. Meantime, we denote $I(f)$ as the index series of the leading term $lt(f)$ of the polynomial $f$ under the descending order, i.e. $I(f)=I(lt(f))$.

Now we are in the position of determining the relations about the generators $[M(1,\overline{1})],[M(2,\overline{0})]$ and
$[M(p^{l}+1,\overline{0})] (\forall 1\leq l\leq m-1)$ of the Nakayama truncated algebra $KZ_{n}/J^{d}$.

{\bf Step $1$.} Determine the relations about $[M(1,\overline{1})]$
and $[M(2,\overline{0})]$.

By Lemma \ref{lem6.1} (i),  the relation $[M(1,
\overline{1})]^{n}=[M(1,\overline{0})]$ holds, which corresponds to
$y^{n}-1=0$.

By Lemma \ref{lem6.1} (iii), the relation
$([M(2,\overline{0})]-[M(1,\overline{1})]-[M(1,\overline{0})])[M(p,\overline{0})]=0$
holds, which corresponds to $(z-y-1)g(p,1,z)=0$.

Now define
$$g_{0}(y,z,w_{1},\cdots,w_{m-1})=y^{n}-1,\;\;\;\;\;\;\;g_{1}(y,z,w_{1},\cdots,w_{m-1})=(z-y-1)g(p,1,z).$$
 Note the leading terms of
$g_{0}(y,z,w_{1},\cdots,w_{m-1})$ and $g_{1}(y,z,w_{1},\cdots,w_{m-1})$ are $y^{n}$ and $z^{p}$, respectively.

{\bf Step $2$.} Determine the relations about
$[M(p^{l}+1,\overline{0})]$ for $1\leq l\leq m-1$.

Firstly, choose $k=1,2,\cdots,p-1$ in Lemma \ref{thm7.1} successively, we obtain the following equalities:

$[M(p^{l}+1,\overline{0})]^{2}=[M(2p^{l}+1,\overline{0})]+\cdots$

$[M(p^{l}+1,\overline{0})][M(2p^{l}+1,\overline{0})]=[M(3p^{l}+1,\overline{0})]+\cdots$

$.............$

$[M(p^{l}+1,\overline{0})][M((p-2)p^{l}+1,\overline{0})]=[M((p-1)p^{l}+1,\overline{0})]+\cdots$

$[M(p^{l}+1,\overline{0})][M((p-1)p^{l}+1,\overline{0})]=[M(p^{l+1},\overline{0})]+[M(p^{l+1},\overline{1})]+\cdots$

Then, from the first formula, we have
$[M(2p^{l}+1,\overline{0})]=[M(p^{l}+1,\overline{0})]^{2}-\cdots$,
by substituting it into the second formula, we obtain
$[M(3p^{l}+1,\overline{0})]=[M(p^{l}+1,\overline{0})]^{3}-\cdots$,
then by substituting $[M(3p^{l}+1,\overline{0})]$ into the third
formula, $\cdots\cdots$, by continuing this process, we obtain that
$[M((p-1)p^{l}+1,\overline{0})]=[M(p^{l}+1,\overline{0})]^{p-1}-\cdots$.
By substituting it into the last formula, we finally obtain a
polynomial $g_{l+1}$ such that the formula
$$g_{l+1}([M(1,\overline{1})], [M(2,\overline{0})], [M(p+1,\overline{0})],
[M(p^{2}+1,\overline{0})],\cdots,[M(p^{l}+1,\overline{0})])=0.$$
holds.

Then after substituting $[M(1,\overline{1})]$ by $y$,
$[M(2,\overline{0})]$ by $z$ and $[M(p^{s}+1,\overline{0})]$ by
$w_{s}$, $1\leq s\leq l$ respectively, we have the polynomials
$g_{l+1}(y,z,w_{1},\cdots,w_{l})$ for $1\leq l\leq m-1.$

{\bf Step $3$.} Finally, we construct the ideal $I$ of
$\mathds{Z}[y,z,w_{1},\cdots,w_{m-1}]$ generated by the polynomials
$g_{i}(y,z,w_{1},\cdots,w_{m-1}), 0\leq i\leq m$.

The following observations are crucial.

\begin{Lem}\label{lem7.4}
\rm {Given $l,s$ satisfying $1\leq l\leq m, 2\leq s\leq p$, let
$p\leq u\leq sp^{l}, 0\leq r\leq n-1$ and denote by
$f_{u,\overline{r}}$ the corresponding polynomial of
$[M(u,\overline{r})]$ generated by
$[M(1,\overline{1})],[M(2,\overline{0})]$ and
$[M(p^{l}+1,\overline{0})],1\leq l\leq m-1$ in Theorem \ref{thm6.3}.
  Then $f_{u,\overline{r}}$ is a polynomial in $\mathds{Z}[y,z,w_{1},\cdots,w_{l}]$, however, no $w_{l}^{t} ~ (t\geq s)$ appears in any monomials of  $f_{u,\overline{r}}$.}
\end{Lem}

\begin{proof}
It is obvious that the polynomial $f_{u,\overline{r}}$ is a polynomial in
$\mathds{Z}[y,z,w_{1},\cdots,w_{l}]$ according to Corollary \ref{lem6.4}
and the correspondence among  $[M(1,\overline{1})],
[M(2,\overline{0})]$, $[M(p^{l}+1,\overline{0})]$ and $y$, $z$,
$w_l$ for $1\leq l\leq m-1$.

 Now we only need to prove that $w_l^t \; (t\geq s)$ will not appear in any monomials of $f_{u,\overline{r}}$.
 It is sufficient to prove the result for $[M(u,\overline{0})]$ instead of $[M(u,\overline{r})]$ since $I(f_{u,\overline{r}})=I(f_{u,\overline{0}})+(0,0,\cdots,0,r)$ for $u\leq p$. We will prove the result by induction on $s$ satisfying $2\leq s \leq p$.

When $s=2$, we have $p\leq u\leq 2p^{l}$. We will claim that
$I(f_{u,\overline{0}})<I(w_{l}^{2})$.

In the case $p\leq u\leq p^{l}$, the result follows from Corollary
\ref{lem6.4} since $w_{l}$ does not appear in the polynomial
$f_{u,\overline{0}}$.

 In the case $p^{l}+1\leq u\leq 2p^{l}$, we will use the induction on $u$. Firstly if $u=p^{l}+1$, then $I(f_{u,0})=I(w_{l})<I(w_{l}^{2})$. Now we assume the result holds for $u'$, $p^{l}+1\leq u'< u\leq 2p^{l}$. Since in this case $u-1$ is not a power of $p$, we have $[M(u,\overline{0})]=[M(u-w,\overline{0})][M(w+1,\overline{0})]-[M_{2}]-\cdots -[M_{t}]$ for some $w$ such that $1\leq w\leq u-2$ and $\binom{u-1}{w}\neq 0$ in $K$, due to the proof of Lemma \ref{lem6.2}. Moreover, $M_{k}\cong M(i_{k},\overline{j_{k}})$ for any $2\leq k\leq t$ with $1\leq i_{k}, j_{k}\leq u-2$.

By the induction hypotheses on $u$, we have
$I(f_{i_{k},\overline{j_{k}}})<I(w_{l}^{2})$ for $2\leq k\leq t$.
Thus it remains to prove
$I(f_{u-w,\overline{0}}f_{w+1,\overline{0}})<I(w_{l}^{2})$. Since
$(u-w)+(w+1)=u+1\leq 2p^{l}+1$, we obtain that either $u-w<p^{l}+1$
or $w+1<p^{l}+1$. Assume the former case is true, then by Corollary
\ref{lem6.4}, we have $I(f_{u-w,\overline{0}})<I(w_{l})$, i.e.,
$w_{l}$ does not appear in $f_{u-w,0}$. Again by the induction
hypotheses on $u$, since $w+1<u$, we have
$I(f_{w+1,\overline{0}})<I(w_{l}^{2})$. Therefore,
$$I(f_{u-w,\overline{0}}f_{w+1,\overline{0}})=I(f_{u-w,\overline{0}})+I(f_{w+1,\overline{0}})<I(w_{l}^{2}).$$ In
the latter case, i.e. $w+1<p^{l}+1$, the similar discussion can be
given.

Hence the result is proved when $s=2$.

Now we assume for all integers $s'$ with $2\leq s'<s$, the result is
true. Then if $p\leq u\leq (s-1)p^{l}$, by the induction hypotheses
on $s$, we have $I(f_{u,\overline{0}})<I(w_{l}^{s-1})<I(w_{l}^{s})$,
which means the result to be true in this case.

Now we consider the result in the case $(s-1)p^{l}+1\leq u\leq sp^{l}$, using the induction on $u$ simultaneously.

Actually, for any $u, (s-1)p^{l}+1\leq u\leq sp^{l}$, since $u-1$ is
not a power of $p$, thus we again have the equation
$[M(u,\overline{0})]=[M(u-w,\overline{0})][M(w+1,\overline{0})]-[M_{2}]-\cdots
-[M_{t}]$ for some $w$ satisfying $1\leq w\leq u-2$ and
$\binom{u-1}{w}\neq 0$ in $K$. In this case the remaining proof is
similar with that in the case $s=2$ by replacing $I(w^2_l)$ with
$I(w^s_l)$. Indeed, we will use the fact that $(u-w)+(w+1)=u+1\leq
sp^{l}+1$ and the induction hypotheses on $u$.
\end{proof}

\begin{Lem}\label{lem7.10}
{\rm Under the order $y<z<w_{1}<w_{2}<\cdots<w_{m-1}$, the leading terms of the polynomials $g_{0}, g_{1},g_{2},\cdots,g_{m}$ constructed above are $y^{n}, z^{p},w_{1}^{p},\cdots,w_{m-1}^p$ respectively.}
\end{Lem}

\begin{proof} The leading terms of $g_{0}, g_{1}$ are $y^{n}, z^{p}$
 clearly due to the definition of $g_{0}, g_{1}$.

 Due to {\bf Step $2$} in this section, for $2\leq s\leq m$,   we have a sequence of the equalities in the form $[M(p^{s-1}+1,\overline{0})]^{p}=\sum[M(r,\overline{t})]$ with $r\leq p^{s}$,
 which correspond respectively to the equalities $w_{s-1}^{p}=\sum f_{r,\overline{t}}$ satisfying $r\leq p^{s}$. Then according to the construction of $g_s$ in
 {\bf Step $2$}, the polynomial $g_s=w_{s-1}^{p}-\sum f_{r,\overline{t}}$. By Theorem \ref{thm6.3} and Lemma \ref{lem7.4},
 all the terms $f_{r,\overline{t}}$ with $r\leq p^{s}$ on the right side of the equations can be represented by a polynomial with variables $y,z,w_{1},\cdots,w_{s-1}$ and
 $I(f_{r,\overline{t}})<I(w_{s-1}^{p})$. Thus, the result follows at once.
\end{proof}

With the above constructions, we have the following main theorem.

\begin{Thm}\label{thm7.5}
{\rm Let {\rm char}$K=p$ and $n\geq d=p^m$, then the representation
ring $r(KZ_{n}/J^{d})$ of the Nakayama truncated algebra
$KZ_{n}/J^{d}$ is isomorphic to the quotient ring
$\mathds{Z}[y,z,w_{1},\cdots,w_{m-1}]/I$, induced by the ring
epimorphism $$\phi:
\mathds{Z}[y,z,w_{1},\cdots,w_{m-1}]\longrightarrow
r(KZ_{n}/J^{d})$$ satisfying
$\phi(1)=[M(1,\overline{0})],\phi(y)=[M(1,\overline{1})],\phi(z)=[M(2,\overline{0})]$ and
 $\phi(w_{l})=[M(p^{l}+1,\overline{0})], \forall \;1\leq l\leq m-1$,
where the ideal $I$ is generated by the polynomials
$g_{i}(y,z,w_{1},\cdots,w_{m-1}), 0\leq i\leq m$ defined as
above.}
\end{Thm}

\begin{proof}
Firstly, by the construction of the ideal $I$, the ring epimorphism $\phi$ can induce a
ring epimorphism
\begin{equation}\label{isom}
\bar{\phi}:\mathds{Z}[y,z,w_{1},\cdots,w_{m-1}]/I\rightarrow r(KZ_{n}/J^{d})
 \end{equation}
such that $\bar{\phi}(\bar{v})=\phi(v)$ for all $v \in
 \mathds{Z}[y,z,w_{1},\cdots,w_{m-1}]$.

 In order to prove that $\bar{\phi}$ is a ring isomorphism, it is enough to claim that the ranks of two sides of the epimorphism $\bar{\phi}$ in
(\ref{isom}) are equal.

By Lemma \ref{lem7.10}, the leading terms of the polynomials $g_{0}, g_{1},g_{2},\cdots,g_{m}$ are $y^{n}, z^{p},w_{1}^{p},\cdots,w_{m-1}^p$, respectively. Thus it is easily seen that
$$B=\{y^{i}z^{j}w_{1}^{l_{1}}\cdots w_{m-1}^{l_{m-1}}~ |~ 0\leq i<n,0\leq j<p,0\leq l_{k}<p, k=1,\cdots,m-1\}$$ forms a $\mathds{Z}$-basis of $\mathds{Z}[y,z,w_{1},\cdots,w_{n-1}]/I$.

Now we can calculate that the ranks of two sides of the
 epimorphism $\bar{\phi}$ in (\ref{isom}) are both $np^m=nd$.
  Actually, the rank of the left side is $\mid B\mid = np^m$ since $B$ is a $\mathds{Z}$-basis of $\mathds{Z}[y,z,w_{1},\cdots,w_{n-1}]/I$; and the rank of the right side is $nd$ by the definition of the representation ring $r(KZ_{n}/J^{d})$. Finally the conclusion follows by the fact that two free abelian groups are isomorphic if and only if both of them have the same rank.
\end{proof}

Finally, we will give the following example to illustrate our result precisely.

\begin{Ex}\label{ex7.4}
{\rm Let $n=10$, then the representation ring $r(KZ_{10}/J^{d})$ is shown as follows:

Case $1$. When {\rm char}$~K_{1}=2$, then

$r(K_{1}Z_{10}/J^{2})\cong \mathds{Z}[y,z]/(y^{10}-1,(z-y-1)z)$.

$r(K_{1}Z_{10}/J^{4})\cong
\mathds{Z}[y,z,w_{1}]/(y^{10}-1,(z-y-1)z,(w_{1}-y)(w_{1}+y-z-yz))$.

$r(K_{1}Z_{10}/J^{8})\cong
\mathds{Z}[y,z,w_{1},w_{2}]/(y^{10}-1,(z-y-1)z,(w_{1}-y)(w_{1}+y-z-yz),w_{2}^{2}-y^{4}-z(w_{1}-y)(1+y)(w_{2}+2y^2-y^{2}z-yz))$.

Case $2$. When {\rm char}$~K_{2}=3$, then

$r(K_{2}Z_{10}/J^{3})\cong \mathds{Z}[y,z]/(y^{10}-1,(z-y-1)(z^{2}-y))$.

$r(K_{2}Z_{10}/J^{9})\cong
\mathds{Z}[y,z,w_{1}]/(y^{10}-1,(z-y-1)(z^{2}-y),(w_{1}+y+y^{2}+yz-z^{2}-yz^{2})(w_{1}^{2}-y^{3}-2yzw_{1}+y^{2}z^{2}))$.

Case $3$. When {\rm char}$~K_{3}=5$, then

$r(K_{3}Z_{10}/J^{5})\cong \mathds{Z}[y,z]/(y^{10}-1,(z-y-1)(z^{4}-3yz^{2}+y^{2}))$.

Case $4$. When {\rm char}$~K_{4}=7$, then

$r(K_{4}Z_{10}/J^{7})\cong
\mathds{Z}[y,z]/(y^{10}-1,(z-y-1)(z^{6}-5yz^{4}+6y^{2}z^{2}-y^{3}))$.
 }
\end{Ex}

At the end of this section, we apply the polynomial characterization in
Theorem \ref{thm7.5} to give some discussions about the isomorphic problem
on representation rings, that is, what conditions should be satisfied of two
Nakayama truncated algebras $KZ_{n_{1}}/J^{d_{1}}$ and
$KZ_{n_{2}}/J^{d_{2}}$ such that $r(KZ_{n_{1}}/J^{d_{1}})\cong
r(KZ_{n_{2}}/J^{d_{2}})$ as rings?

In fact, we have to replace $r(KZ_{n}/J^{d})$ by its {\em complexified representation algebra} $R(KZ_{n}/J^{d})=\mathds{C}\otimes_{\mathds{Z}} r(KZ_{n}/J^{d})$ and then discuss the conditions for $R(KZ_{n_{1}}/J^{d_{1}})\cong R(KZ_{n_{2}}/J^{d_{2}})$ holds.

Recall that for a field $K$, a positive integer $t$ and the polynomial algebra $K[x_{1},\cdots,x_{t}]$, let $V$ be a subset of $K^t$,
then denote $I(V)=\{f\in K[x_{1},\cdots,x_{t}]|f(v)=0$ for all $v \in V\}$ which is an ideal of $K[x_{1},\cdots,x_{t}]$, called as the {\em corresponding ideal} of $V$; conversely for an ideal $I$ of $K[x_{1},\cdots,x_{t}]$, denote $V(I)=\{v\in K^{t}|f(v)=0$ for all $f \in I\}\subset K^t$, which is called an {\em algebraic set} of $K^{t}$ corresponding to $I$.

Moreover, a subset $V\subset K^t$ is called an \emph{affine variety} if it is an irreducible closed subset of $K^{t}$ under the meaning of Zariski topology; and an ideal $I$ of $K[x_{1},\cdots,x_{t}]$ is called an \emph{radical ideal} if $\sqrt{I}=I$ for  $\sqrt{I}=\{f|f^{s}\in I$ for some integer $s>0\}$.

It is well-known that there exists a one-to-one correspondence between affine varieties $V$ of $K^{t}$ and radical ideals $I$ of $K[x_{1},\cdots,x_{t}]$.

Furthermore, we will denote $K[V]$ to be the \emph{homogeneous coordinate ring} of the affine variety $V$. For the details, see \cite{[15]},\cite{[23]}.

\begin{Lem}{\rm (Corollary 4.5, \cite{[23]})}\label{lem7.5}
{\rm For two affine varieties $V,W$ in $K^{t}$, a polynomial map $f:V\rightarrow W$ is an isomorphism if and only if  the dual map $f^{*}:K[W]\rightarrow K[V]$ is an isomorphism.}
\end{Lem}

We know $K[V]\stackrel{{\rm def.}}{=}\{f:V\rightarrow K|f$ is a polynomial function$\}$; for a polynomial function $g\in K[W]$, $f^*$ is defined by $f^{*}(g)=g\circ f$.
  For two finite sets $V,W$ in $K^{t}$, there exists a polynomial isomorphism between $V$ and $W$ if and only if $|V|=|W|$. Then we have the following:
\begin{Cor}\label{Cor8.2}
{\rm  Let {\rm char}$K=p$ and $n_i\geq d_i=p^{m_i}, i=1,2$
satisfying $n_{1} d_{1}=n_{2} d_{2}$ and $I,I'$ be the ideals of
$\mathds Z[y,z,w_1,\cdots,w_{m_1-1}]$ and $\mathds
Z[y,z,w_1,\cdots,w_{m_2-1}]$ respectively given in Theorem
\ref{thm7.5} in the process of polynomial characterization of the
representation rings $r(KZ_{n_{1}}/J^{d_{1}})$,
$r(KZ_{n_{2}}/J^{d_{2}})$. If furthermore both $I$ and $I'$ are
radical ideals, then their complexified representation algebras are
isomorphic, that is, $R(KZ_{n_{1}}/J^{d_{1}})\cong
R(KZ_{n_{2}}/J^{d_{2}})$.}
\end{Cor}
\begin{proof}
Denote $N=n_{1}d_{1}=n_{2}d_{2}$, then the algebraic sets $V(I),
V(I')$ have both $N$ discrete points, i.e. $N=|V(I)|=|V(I')|$. In
fact, by the Grobner basis theory, or in brief, since the degree of
$g_{0},g_{1},g_{2},\cdots,g_{m_{1}}$ in Theorem \ref{thm7.5} are
$n_{1},p,p,\cdots,p$ respectively under the order
$y<z<w_{1}<w_{2}<\cdots<w_{m_{1}-1}$, it follows that the algebraic
set $V(I)$ has $n_{1}\cdot p \cdot p \cdots p=n_{1}\cdot d_{1}=N$
discrete points. So is $V(I')$, too.

It follows that  $N=|V(\sqrt{I})|=|V(I)|=|V(I')|=|V(\sqrt{I'})|<\infty$, and then there exists a polynomial isomorphism $f$ from $V(\sqrt{I})$ to  $V(\sqrt{I'})$, which means by Lemma \ref{lem7.5} there exists a $K$-algebra isomorphism $f^*$, as the dual of $f$, from the homogeneous coordinate ring $\mathds{C}[V(\sqrt{I'})]$ to $\mathds{C}[V(\sqrt{I})]$.

Since both $I$ and $I'$ are radical ideals, we have
$\mathds{C}[V(\sqrt{I})]\cong \mathds{C}[y,z,w_1,\cdots,w_{m_1-1}]/I(V(\sqrt{I}))=\mathds{C}[y,z,w_1,\cdots,w_{m_1-1}]/\sqrt{I}= \mathds{C}[y,z,w_1,\cdots,w_{m_1-1}]/I\cong R(KZ_{n_{1}}/J^{d_{1}})$, similarly the isomorphism $\mathds{C}[V(\sqrt{I'})]\cong R(KZ_{n_{2}}/J^{d_{2}})$ holds.
 Hence, we get that $R(KZ_{n_{1}}/J^{d_{1}})\cong R(KZ_{n_{2}}/J^{d_{2}})$, which completes the proof.
\end{proof}

From this result, it is possible for one to find some different Nakayama truncated algebras whose complexified representation algebras are isomorphic together each other.

However, in general, the ideal $I$ in Theorem \ref{thm7.5} is not always an radical ideal. For instance, if we choose $n=10, d=p=2$, then in Example \ref{ex7.4} we have obtained that $r(K_{1}Z_{10}/J^{2})\cong \mathds{Z}[y,z]/(y^{10}-1,(z-y-1)z)$, that is, the ideal $I=(y^{10}-1,(z-y-1)z)$; but, through calculating by using the Maple $17$ software, it is easy to see that $\sqrt{I}=(y^{10}-1,(z-y-1)z,\frac{y^{10}-1}{y+1}z)\supsetneqq I$, which says that $I$ is not radical ideal.
\bigskip

\section{Shift rings and derived rings of Nakayama truncated algebras}

\subsection{{ The shift ring $sh(KZ_{n}/J^{d})$ }}

For the Nakayama truncated algebra $H=KZ_{n}/J^{d}$ with ${\rm char}K=p, n\geq d=p^m, m>0$, recall that the indecomposable modules
 $M(i,\overline{j})=P_{\overline{j}}/{\rm rad}^{i}P_{\overline{j}},0\leq i\leq n-1,1\leq j\leq d$, $s\in \mathds{Z}$, where $P_{\overline{j}}$
 is the indecomposable projective $H$-module at the vertex $\overline{j}$.
In order to calculate the shift ring $sh(KZ_{n}/J^{d})$, it suffices
to compute the structure constants $k_{i,j,i',j'}^{i'',j''}$ in the
decomposition $$M(i,\overline{j})^{\bullet}[s]\otimes
M(i',\overline{j'})^{\bullet}[s']\cong \sum \limits_{0\leq i''\leq
n-1, 1\leq j''\leq d}k_{i,j,i',j'}^{i'',j''}M(i'',\overline{j''})^{\bullet}[s+s']$$ for
$0\leq i,i',i'' \leq n-1,1\leq j,j',j''\leq d, s,s'\in \mathds{Z}$.
However, by the formula (\ref{addshift}) in Section $2$, these
$k_{i,j,i',j'}^{i'',j''}$ are just the structure constants in the
decomposition $$M(i,\overline{j})\otimes M(i',\overline{j'})\cong
\sum \limits_{0\leq i''\leq n-1, 1\leq j''\leq d}
k_{i,j,i',j'}^{i'',j''}M(i'',\overline{j''})$$ for $0\leq
i,i',i''\leq n-1,1\leq j,j',j''\leq d$ in the module category, which
we have already known by Lemma \ref{lem6.1} and Lemma \ref{thm7.1}.
Hence we can give the generators and relations of the shift ring
$r^{sh}(KZ_{n}/J^{d})$ as follows.

\begin{Prop}\label{prop10.3}
{\rm Let ${\rm char}K=p, n\geq d=p^m, m>0$. Then

(i)~ The shift ring $sh(KZ_{n}/J^{d})$ is generated by
$[M(1,\overline{1})^{\bullet}[r]],[M(2,\overline{0})^{\bullet}[r]]$
and $[M(p^{l}+1,\overline{0})^{\bullet}[r]]$, where $r\in
\mathds{Z}, 1\leq l\leq m-1$.

(ii)~ The relations in shift ring $sh(KZ_{n}/J^{d})$ can be determined by the following formulas:
\[  \left\{
\begin{array}{ll}
M(i,\overline{j})^{\bullet}[r]\otimes M(1,\overline{0})^{\bullet}[s]\cong M(1,\overline{0})^{\bullet}[s]\otimes M(i,\overline{j})^{\bullet}[r] \cong M(i,\overline{j})^{\bullet}[r+s], \\
M(1,\overline{1})^{\bullet}[r]^{\otimes n}\cong M(1,\overline{0})^{\bullet}[n r], \\
M(2,\overline{0})^{\bullet}[r]\otimes M(t,\overline{0})^{\bullet}[s]\cong  M(t+1,\overline{0})^{\bullet}[r+s]\oplus M(t-1,\overline{1})^{\bullet}[r+s] ~~~~\forall t\geq 2, p\nmid t, \\
M(2,\overline{0})^{\bullet}[r]\otimes M(t,\overline{0})^{\bullet}[s]\cong M(t,\overline{0})^{\bullet}[r+s]\oplus M(t,\overline{1})^{\bullet}[r+s] $~~for ~all~$ t>0, p|t, \\
\end{array}
\right.
\]
$M(p^{l}+1,\overline{0})^{\bullet}[r]\otimes
M(kp^{l}+1,\overline{0})^{\bullet}[s] \cong \left\{\begin{array}{ll}
W_{1}[r+s], &\mbox{if~ $k\equiv -1({\rm mod}~p)$} \\
W_{2}[r+s], &\mbox{if~ $k\equiv 0({\rm mod}~p)$} \\
W_{3}[r+s],  &\mbox{otherwise} \\
\end{array}\right.$\\
where $r,s\in \mathds{Z}$ and $W_{1},W_{2},W_{3}$ are defined in Lemma \ref{thm7.1}.

(iii) Let $I'$ be the ideal of $\mathds{Z}[y[r],z[r],w_{l}[r]]_{r\in
\mathds{Z}, 1\leq l\leq m-1}$ generated by the relations which are
given from (ii) through replacing
$[M(1,\overline{1})^{\bullet}[r]],[M(2,\overline{0})^{\bullet}[r]],[M(p^{l}+1,\overline{0})^{\bullet}[r]]$
by the indeterminates $y[r],z[r],w_{l}[r], r\in \mathds{Z}, 1\leq
l\leq m-1$, respectively. Then there is the following ring
isomorphism:  $$sh(KZ_{n}/J^{d})\cong
\mathds{Z}[y[r],z[r],w_{l}[r]]_{r\in \mathds{Z}, 1\leq l\leq
m-1}/I'.$$ }
\end{Prop}

\begin{proof}
 (i) follows from Theorem \ref{thm6.3} and (ii) from Lemma \ref{lem6.1} and Lemma \ref{thm7.1}. Additionally, (iii) follows from Theorem \ref{thm7.5}, concretely, we can define a ring homomorphism $$\phi:\mathds{Z}[y[r],z[r],w_{l}[r]]_{r\in \mathds{Z}, 1\leq l\leq m-1}\rightarrow sh(KZ_{n}/J^{d})$$ by $\phi(y[r])=[M(1,\overline{1})^{\bullet}[r]]$, $\phi(z[r])=[M(2,\overline{0})^{\bullet}[r]]$ and $\phi(w_{l}[r])=[M(p^{l}+1,\overline{0})^{\bullet}[r]]$, the remaining is similar.
\end{proof}

\subsection{{The derived ring $dr(KZ_{n}/J^{2})$} }

So far, it is difficult to obtain all indecomposable objects in the
bounded derived category $D^{b}(H)$, in general, for an arbitrary
representation-finite $K$-Hopf algebra $H$, even for the Nakayama
truncated algebra $H=KZ_{n}/J^{d}$.

When $d=2$, however, we can list all the indecomposable complexes of
$D^{b}(H)$ by using the method introduced in \cite{[3]} since
$H=KZ_{n}/J^{2}$ is a gentle algebra, or more generally, an
elementary algebra with 2-nilpotent radical. All the notations and
terminologies can be found in \cite{[3]}.

\subsubsection{ Construction of the indecomposable objects in $D^{b}(KZ_{n}/J^{2})$.}

\begin{Lem}
{\rm The set of indecomposable objects of $D^{b}(KZ_{n}/J^{2})$ is
given as
$${\rm ind}(D^{b}(KZ_{n}/J^{2}))=\{P^{\bullet}(i,j)[l]~ |~ -\infty \leq j\leq i<+\infty, l\in
\mathds{Z}\},$$
 where \[ P^{\bullet}(i,j)_{a}= \left\{
\begin{array}{ll}
P_{\overline{a}},      &\mbox{if $i\leq a\leq j$, }  \\
0,      &\mbox{otherwise,}
\end{array}
\right.\]

\[ P^{\bullet}(i,j)_{a+1\rightarrow a}= \left\{
\begin{array}{ll}
P_{\overline{a+1}}\rightarrow P_{\overline{a}},     &\mbox{if $i\leq a\leq j$,}  \\
0,      &\mbox{otherwise.}
\end{array}
\right.\] and the morphism $P_{\overline{a+1}}\rightarrow
P_{\overline{a}}$ is defined by the action of arrows. }
\end{Lem}

\begin{proof}
Our proof is dependent on Theorem 3.11 in \cite{[3]}, where the
composition of arrows is from left to right, which is different with
the composition in this paper. So, in fact, we will use here the
dual conclusion for the converse composition, that naturally follows
from that in \cite{[3]}.

Firstly, it is easy to see that the minimal grading covering
$\widetilde{Q}$ of $Q$ is just the infinite Dynkin graph
$A_{\infty}$ with  the universal covering map sends $m$ to $m$(under
the relation of modulo $n$), see Figure $7$. We may further denote
the vertices of $A_{\infty}$ as $\cdots -2, -1, 0, 1, 2, \cdots$.

\begin{figure}[h] \centering
  \includegraphics*[135,565][421,708]{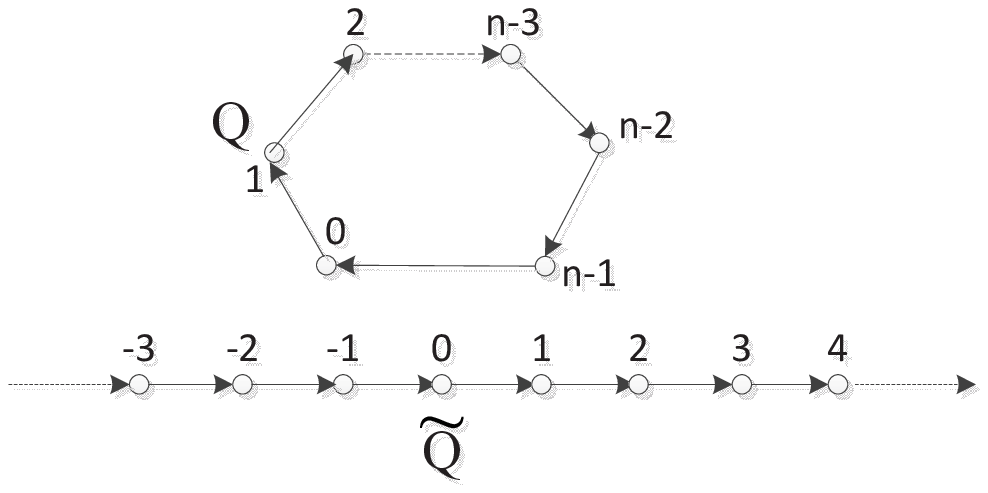}

 {\rm Figure 7 }
\end{figure}

Since $KQ/J^{2}$ is an elementary algebra with 2-nilpotent radical,
by Theorem 3.11 in \cite{[3]} there is a functor $\mathscr{F}: {\rm
rep}^{-,p}(\widetilde{Q})\rightarrow D^{b}(KQ/J^{2})$ which
preserves isomorphism classes and indecomposability, where ${\rm
rep}^{-,p}(\widetilde{Q})$ denotes the fully subcategory of ${\rm
rep}(\widetilde{Q})$ consisting of the bounded-above truncated
injective representations. Moreover, by Lemma 3.4 in \cite{[3]},
there exists an indecomposable object $N$ in ${\rm
rep}^{-,p}(\widetilde{Q})$ and an integer $n$ such that $M \cong
\mathscr{F}(N)[n]$ for any indecomposable object in
$D^{b}(KZ_{n}/J^{2})$. Hence, in order to characterize the
indecomposable objects of $D^{b}(KZ_{n}/J^{2})$, it suffices to
compute the indecomposable objects of ${\rm
rep}^{-,p}(\widetilde{Q})$.

Secondly, as the quiver $\widetilde{Q}=A_{\infty}$ with linear
orientation, hence all the indecomposable bounded-above objects in
${\rm rep}^{-,p}(\widetilde{Q})$ are $\{L(i,j)\}_{-\infty \leq j\leq
i <+\infty }$, where $L(i,j)$ for any $i,j$ satisfies that {\[
L(i,j)_{a}= \left\{
\begin{array}{ll}
K,      &\mbox{$(j\leq a\leq i)$,}  \\
0,      &\mbox{otherwise,}
\end{array}
\right. L(i,j)_{a\rightarrow a+1}= \left\{
\begin{array}{ll}
id_{K},      &\mbox{$(j\leq a\leq i-1)$,}  \\
0,      &\mbox{otherwise.}
\end{array}
\right.\]}

Finally, we consider the images of $L(i,j)$ under the action of the
functor $\mathscr{F}$. By the definition of $L(i,j)$ above, we have
that $\mathscr{F}(L(i+1,j+1))=P^\bullet(i,j)$, where
$P^\bullet(i,j)$ are defined in this lemma. Because $P^\bullet(i,j)$
has bounded homology for any
 $-\infty\leq j\leq i<+\infty$, so each such $L(i,j)$ is included in ${\rm rep}^{-,p}(\widetilde{Q})$
by Proposition 3.8 in \cite{[3]}. Therefore, the set of all
indecomposable objects of $D^{b}(KZ_{n}/J^{2})$ is just the set
$\{P^{\bullet}(i,j)[l]\}_{-\infty\leq j\leq i<+\infty, l\in
\mathds{Z}}$ by Proposition 3.1 and Theorem 3.11 in \cite{[3]}.
\end{proof}

In particular, we have $P^{\bullet}(i, -\infty)=\cdots\rightarrow
0\rightarrow P_{\overline{i}}\rightarrow
P_{\overline{i-1}}\rightarrow \cdots$. Note that if we choose $j=i$,
then $\mathscr{F}(L(i,i))$ is the stalk complex $P_{\overline{i}}$
in the $i$-th position; and choose $j=-\infty$, then
$\mathscr{F}(L(i, -\infty))$ is isomorphic to the stalk complex
$S_{\overline{i+1}}$ in the $i$-th position, since
$S_{\overline{i+1}}[i]$ is quasi-isomorphic to $P^{\bullet}(i,
-\infty)$ in the complex category $Ch(H)$.

\subsubsection{Structure constants of the derived ring $dr(KZ_{n}/J^{2})$ }

The derived ring $dr(KZ_{n}/J^{2})$ is generated by all objects in
${\rm ind}(D^{b}(H))=\{(P^{\bullet}(i,j))[l]~ |~ -\infty\leq j\leq
i<+\infty,l\in \mathds{Z}\}$.  As a subring of $dr(KZ_{n}/J^{2})$,
the structure constants of $sh(KZ_{n}/J^{2})$ associated with the
stalk complexes are already known from that of $r(KZ_{n}/J^{2})$ in the last subsection, or
Section 2. So it remains to consider the structure constants
associated with the tensor product $P^{\bullet}(i',j'))[l]\otimes
P^{\bullet}(i,j))[l']$ where at least one of
$P^{\bullet}(i',j'))[l], P^{\bullet}(i,j))[l']$ is not a stalk
complex. For simplify, we may further assume that $l=l'=0$. In order
to this aim, we first calculate the homological groups of
$P^{\bullet}(i',j'))\otimes P^{\bullet}(i,j))$,

\begin{Prop}\label{prop10.5}
{\rm  For $j,j',s,s'\in \mathds{Z}$ and $-\infty<s'\leq s$, let
\begin{eqnarray*}
\mathcal L &   =   & P^{\bullet}(j'+s'-1,j')\otimes
P^{\bullet}(j+s-1,j).
 \end{eqnarray*}
  Then the homological groups $H_{m}(\mathcal L)$ of $\mathcal L$ for all $m\in
  \mathbb{Z}$ are listed as follows:
    \begin{equation}\label{formulaforR}
H_{m}(\mathcal L)= \left\{
\begin{array}{ll}
S_{\overline{j+j'+s+s'}},   &\mbox{if $m=j+j'+s+s'-2, s'\neq 1$;}  \\
S_{\overline{j+j'+s}},   &\mbox{if $m=j+j'+s-1, s\neq s', s'\neq 1$;}  \\
S_{\overline{j+j'+s'}},   &\mbox{if $m=j+j'+s'-1, s\neq s', s'\neq 1$;}  \\
S_{\overline{j+j'+s}}\oplus S_{\overline{j+j'+s'}},   &\mbox{if $m=j+j'+s-1, s=s'\neq 1$;}  \\
P_{\overline{j+j'+s}},   &\mbox{if $m=j+j'+s-1, s\neq s'=1$;}  \\
S_{\overline{j+j'}},   &\mbox{if $m=j+j', s'\neq 1$;}\\
P_{\overline{j+j'}},   &\mbox{if $m=j+j', s\neq s'=1$;}  \\
P_{\overline{j+j'+1}}\oplus P_{\overline{j+j'}},   &\mbox{if $m=0, s=s'=1$;}  \\
0,      &\mbox{otherwise.}
\end{array}
\right.\end{equation}
   }
\end{Prop}

\begin{proof}
For the notations about spectral sequence, refer to
\cite{[25]}, \S$5.6$. As there exists a double complex $C_{**}$ such
that $\mathcal {L}={\rm Tot}^{\oplus}(C_{**})$, we can calculate the homological groups $H_m(\mathcal L)$ through $C_{**}$ by using the method of spectral
sequences.

When $s\geq s'>1$, since $C_{**}$ is bounded, we obtain a
spectral sequence $\{^{I}E^{r}_{pq}\}$, which converges to
$H_{p+q}(\mathcal L)$ via the natural filtration by columns. More
precisely, the spectral sequence starts with
$\{^{I}E^{0}_{p,q}=C_{p,q}\}$ and the differentials
$d_{pq}^{0}=d_{pq}^{v}$, where $d_{pq}^{v} ~ (\forall p,q\in
\mathbb{Z})$ are the vertical differentials of $C_{**}$ and denote $d^{v}=\{d_{pq}^{v}\}_{p,q\in
\mathbb{Z}}$, $d^{0}=\{d_{pq}^{0}\}_{p,q\in
\mathbb{Z}}$. Now we obtain that
\[
{^{I}E^{1}_{p,q}}=H_{q}^{v}(C_{p,*})= \left\{
\begin{array}{ll}
P_{\overline{j'+p}}\otimes S_{\overline{j+s}},   &\mbox{if $j'\leq p\leq j'+s'-1$, $q=j+s-1$};  \\
P_{\overline{j'+p}}\otimes S_{\overline{j}},   &\mbox{if $j'\leq p\leq j'+s'-1$, $q=j$};  \\
0,      &\mbox{otherwise}.
\end{array}
\right.\]

Furthermore, the maps
$d_{pq}^{1}:{^{I}E^{1}_{p,q}}\rightarrow {^{I}E^{1}_{p-1,q}}$ are
induced by the horizontal differentials $d_{pq}^{h}$ of $C_{**}$.
Hence we obtain that the $2$-piece ${^{I}E^{2}_{p,q}}$ as follows:
\[ {^{I}E^{2}_{p,q}}=\left\{
\begin{array}{ll}
S_{\overline{j'+s'}}\otimes S_{\overline{j+s}}=S_{\overline{j+j'+s+s'}},  &\mbox{if $p=j'+s'-1, q=j+s-1;$}\\
S_{\overline{j'+s'}}\otimes S_{\overline{j}}=S_{\overline{j+j'+s'}},  &\mbox{if $p=j'+s'-1, q=j;$}\\
S_{\overline{j'}}\otimes S_{\overline{j+s}}=S_{\overline{j+j'+s}},  &\mbox{if $p=j', q=j+s-1;$}\\
S_{\overline{j'}}\otimes S_{\overline{j}}=S_{\overline{j+j'}},  &\mbox{if $p=j', q=j;$}\\
0,  &\mbox{otherwise.}
\end{array}
\right.
\]

As $s\geq s'$, there exist no $p,q$ such that both
${^{I}E^{2}_{p,q}}$ and $ {^{I}E^{2}_{p-2,q+1}}$ are non-zeros. Thus
the maps $d_{pq}^{2}: {^{I}E^{2}_{p,q}}\rightarrow
{^{I}E^{2}_{p-2,q+1}}$ are zero maps, which implies that
${^{I}E^{2}_{p,q}}= {^{I}E^{3}_{p,q}}=\cdots {^{I}E^{m}_{p,q}}=
{^{I}E^{\infty}_{p,q}}$. Since the spectral sequence
$\{^{I}E^{r}_{pq}\}$ converges to $H_{p+q}(\mathcal L)$, it follows
that $H_{m}(\mathcal L)$ has a filtration $\{F_{p}H_{m}(\mathcal
L)\}$ such that $F_{p}H_{m}(\mathcal L)/F_{p-1}H_{m}(\mathcal
L)\cong
 {^{I}E^{\infty}_{p,m-p}}$ for any $m \in \mathbb{N}$.

 In case $s\neq s'$, for any $m$, there are at most one $p_{0}$ such that
 $ {^{I}E^{\infty}_{p,m-p}}\neq 0$. Hence by the filtration of $H_{m}(\mathcal L)$, we have $H_{m}(\mathcal L)=^{I}E^{\infty}_{p_{0},m-p_{0}}$.

 In case $s=s'$, by the filtration of $H_{s-1}(\mathcal L)$, there is
an exact sequence $$0\rightarrow S_{\overline{j+j'+s}}\rightarrow
H_{j+j'+s-1}(\mathcal L)\rightarrow S_{\overline{j+j'+s}}\rightarrow
0.$$ Moreover, since
${\rm Ext}^{1}_{H}(S_{\overline{j+j'+s}},S_{\overline{j+j'+s}})=0$, we
have $H_{j+j'+s-1}(\mathcal L)= S_{\overline{j+j'+s}}\oplus
S_{\overline{j+j'+s}}$. Additionally, for $m\neq s-1$, there is at
most one $p_{0}$ such that
 $ {^{I}E^{\infty}_{p_{0},m-p_{0}}}\neq 0$. Hence by the filtration of $H_{m}(\mathcal L)$, we get
$H_{m}(\mathcal L)=^{I}E^{\infty}_{p_{0},m-p_{0}}$.

In a word, when $s\neq s'$, we have
\begin{equation}\label{formulaforR}
H_{m}(\mathcal L)= \left\{
\begin{array}{ll}
S_{\overline{j+j'+s+s'}},   &\mbox{if $m=j+j'+s+s'-2$;}  \\
S_{\overline{j+j'+s}},   &\mbox{if $m=j+j'+s-1$;}  \\
S_{\overline{j+j'+s'}},   &\mbox{if $m=j+j'+s'-1$;}  \\
S_{\overline{j+j'}},   &\mbox{if $m=j+j'$;}\\
0, &\mbox{otherwise;}
\end{array}
\right.\end{equation}
 when $s=s'$, we have
 \begin{equation}\label{formulaforR}
H_{m}(\mathcal L)= \left\{
\begin{array}{ll}
S_{\overline{j+j'+s+s'}},   &\mbox{if $m=j+j'+s+s'-2$;}  \\
S_{\overline{j+j'+s}}\oplus S_{\overline{j+j'+s'}},   &\mbox{if $m=j+j'+s-1$;}  \\
S_{\overline{j+j'}},   &\mbox{if $m=j+j'$;}\\
0, &\mbox{otherwise.}
\end{array}
\right.\end{equation} \\

The proof in the case $s'=1$ is similar following the below explanation.

When $s>s'=1$,  the first piece of
the spectral sequence  $\{{^{I}E_{p,q}^{r}}\}$ can be  calculated as:
 \[ {^{I}E^{1}_{p,q}}=\left\{
\begin{array}{ll}
P_{\overline{j'}}\otimes S_{\overline{j+s}}\cong P_{\overline{j+j'+s}},  &\mbox{if $p=j', q=j+s-1;$}\\
P_{\overline{j'}}\otimes S_{\overline{j}}\cong P_{\overline{j+j'}},  &\mbox{if $p=j', q=j$;}\\
0, &\mbox{otherwise.}
\end{array}
\right.
\]
Since there exists no $p,q$ such that both ${^{I}E^{1}_{p,q}}$
and $ {^{I}E^{1}_{p-1,q}}$ are non-zeros,  the maps $d_{pq}^{1}:
{^{I}E^{1}_{p,q}}\rightarrow {^{I}E^{1}_{p-1,q}}$ are zero maps,
which implies that ${^{I}E^{1}_{p,q}}= {^{I}E^{2}_{p,q}}=\cdots
{^{I}E^{m}_{p,q}}= {^{I}E^{\infty}_{p,q}}$.
Thus, \begin{equation}\label{formulaforR} H_{m}(\mathcal L)=
\left\{
\begin{array}{ll}
P_{\overline{j+j'+s}},   &\mbox{if $m=j+j'+s-1$;}  \\
P_{\overline{j+j'}},   &\mbox{if $m=j+j'$;}  \\
0,  &\mbox{otherwise.}
\end{array}
\right.\end{equation}

Finally, when $s=s'=1$, $\mathcal {L}=P_{\overline{j'}}[j']\otimes
P_{\overline{j}}[j']\cong P_{\overline{j+j'+1}}[j+j']\oplus
P_{\overline{j+j'}}[j+j']$.
\end{proof}

Through Proposition \ref{prop10.5} and some calculations, we can give
the following characterization about objects in $D^{b}(H)$ whose homological
groups are isomorphic to $H_{m}(\mathcal L)$ in two cases.

(I)~ When $s'=1$, an object $\mathcal R$ in $D^{b}(H)$ satisfies
$H_{m}(\mathcal R)\cong H_{m}(\mathcal L)$ for all
$m\in \mathbb{Z}$ if and only if $\mathcal R \cong
P^{\bullet}_{\overline{j+j'}}[j+j']\oplus
P^{\bullet}_{\overline{j+j'+s}}[j+j'+s-1] $.

(II)~ When $s'>1$, an object $\mathcal R$ in $D^{b}(H)$ satisfies
$H_{m}(\mathcal R)\cong H_{m}(\mathcal L)$ for
all $m\in \mathbb{Z}$ if and only if one of the following four cases holds:\\
Case 1: $\mathcal R$ has a direct summand with the form
$S_{\overline{k}}[m]$ for some $0\leq k\leq
n-1$ and $m\in \mathbb{Z}$;\\
Case 2: $\mathcal R\cong \mathcal R_2= P^{\bullet}(j+j'+s'-1,j+j') \oplus
P^{\bullet}(j+j'+s+s'-1,j+j'+s)[-1]$;\\
Case 3: $\mathcal R\cong\mathcal R_3= P^{\bullet}(j+j'+s-1,j+j') \oplus
P^{\bullet}(j+j'+s+s'-1,j+j'+s')[-1]$;\\
Case 4: $\mathcal R\cong\mathcal R_4= P^{\bullet}(j+j'+s+s'-1,j+j')
\oplus P^{\bullet}(j+j'+s-1,j+j'+s')[-1]$, this case happens only if $s>s'$.

Obviously the isomorphic objects in $D^{b}(H)$ must have the isomorphic
homological groups. Therefore, from (I), when $s'=1$, we have the decomposition
\begin{equation}\label{relations'=1}
\mathcal L=P^{\bullet}(j'+s'-1,j')\otimes P^{\bullet}(j+s-1,j)=P^{\bullet}_{j+j'}[j+j']\oplus P^{\bullet}_{j+j'+s}[j+j'+s-1];
\end{equation}
when $s'>1$, the decomposition of $\mathcal L$ must be equal to one of $\mathcal R$'s in the above four
cases. However, the following Proposition \ref{prop10.7} tells us that the decomposition of $\mathcal L$ is impossible equals to the form in Case 1
above in (II).

\begin{Lem}\label{lem10.6}
{\rm A chain map $f: P^{\bullet}(i, -\infty) \rightarrow
P^{\bullet}(i, -\infty)$ is null homotopic if and only if $f=0$.}
\end{Lem}
\begin{proof}
If $f=(f_j)_{j\in \mathbb{Z}}$ is null homotopic, then there exist
$s_j:P_{\overline{j}}\rightarrow P_{\overline{j+1}}$ such that
$f_j=s_{j-1}d_j+d_{j+1}s_j$ for all $j\in \mathbb{Z}$, where $d_j$
denote the differentials of $P^{\bullet}(i, -\infty)$. By the
definition of $H$, it is easy to see that $\rm
Hom_{H}(P_{\overline{i}},P_{\overline{i+1}})=0$, therefore $s_j=0$
for all $j\in \mathbb{Z}$. Hence $f_j=0$ for all $j\in \mathbb{Z}$.
The converse result is trivial.
\end{proof}

\begin{Prop}\label{prop10.7}
{\rm $\mathcal L=P^{\bullet}(j'+s'-1,j')\otimes P^{\bullet}(j+s-1,j)$
contains no direct summand as the form $S_{k}^{\bullet}[m]$ for $0\leq k\leq n-1$ and $m\in \mathbb{Z}$.}
\end{Prop}

\begin{proof}
Otherwise, there exists $m\in \mathbb{Z}$, $0\leq k\leq n-1$ such
that $S_{k}^{\bullet}[m]$ is a direct summand of $\mathcal L$. Since
$\mathcal {K}^{-}(\mathcal {I})\cong D^{b}(H)$, $P^{\bullet}(k-1, -\infty)$ is the minimal injective resolution of
$S_k$ and $\mathcal {K}^{-}(\mathcal {I})$ is the bounded above
homotopy category of injective complexes, which is a subcategory of
homotopy category consisting of bounded above injective complexes, thus
$P^{\bullet}(k-1, -\infty)[m]$ is a direct summand of $\mathcal L$
in $\mathcal {K}^{-}(\mathcal {I})$. Therefore, there exists
$\overline{f}\in \rm Hom_{\mathcal {K}^{-}(\mathcal {I})}(\mathcal
L, P^{\bullet}(k-1, -\infty)[m])$ and $\overline{g}\in \rm
Hom_{\mathcal {K}^{-}(\mathcal {I})}(P^{\bullet}(k-1, -\infty)[m],
\mathcal L)$ such that $\overline{f}\overline{g}=id$ in $\mathcal
{K}^{-}(\mathcal {I})$. Equivalently, there exists $f\in \rm
Hom_{Ch(H)}(\mathcal L, P^{\bullet}(k-1, -\infty)[m])$ and $g\in \rm
Hom_{Ch(H)}(P^{\bullet}(k-1, -\infty)[m], \mathcal L)$ such that
$fg-id$ is null homotopic in $Ch(H)$. Now by Lemma \ref{lem10.6},
$fg-id=0$, so $fg=id$. Then we deduce that $P^{\bullet}(k-1,
-\infty)[m]$ is a direct summand of $\mathcal L$ in $Ch(H)$. But it
is impossible because $P^{\bullet}(j+s-1,j)\otimes
P^{\bullet}(j'+s'-1,j')$ is a bounded complex.
\end{proof}

From Proposition \ref{prop10.7}, we know that the decomposition of
$\mathcal L$ cannot be in Case 1 when $s'>1$.

Additionally, when $s=s'>1$, Case 2 coincides with Case 3 and Case 4 does not appear. Hence in this condition, we have the decomposition:
\begin{equation}\label{s=s'>1}
\mathcal L=P^{\bullet}(j+j'+s'-1,j+j')\oplus
P^{\bullet}(j+j'+s+s'-1,j+j'+s)[-1].
\end{equation}

Therefore, we only need to consider the decomposition of $\mathcal L$ when $s>s'>1$. Unfortunately, so far we cannot determine the decomposition of
$\mathcal L$ in which form of Cases 2, 3, 4 exactly in general. We conjecture that there is only Case 2 can occur in the decomposition of $\mathcal L$, according to several decompositions we have already calculated. In summary, we collect the above discussion into the following conjecture.

\begin{Conj}\label{conj10.4}
{\rm Using the above notations in $D^{b}(KZ_{n}/J^{2})$ for all $j,j',s,s'\in \mathds{Z}$ and $s'\leq s$, we always have the following decomposition of  indecomposable complexes:
\begin{equation}\label{conj}
\mathcal L=P^{\bullet}(j'+s'-1,j')\otimes P^{\bullet}(j+s-1,j)\cong
P^{\bullet}(j+j'+s'-1,j+j') \oplus
P^{\bullet}(j+j'+s+s'-1,j+j'+s)[-1].
\end{equation} }
\end{Conj}

From the above discussion, we know that this conjecture has been affirmed positively except for the case $s>s'>1$.

At last, we give a summary on
the derived ring $dr(KZ_{n}/J^{2})$ as follows.
\begin{Thm}\label{thm10.6}
{\rm For ${\rm char}K=d=2, n\geq 2$, the following statements hold.

(i) The derived ring $dr(KZ_{n}/J^{2})$ is generated by
$[M(1,\overline{1})^{\bullet}[r]],[M(2,\overline{0})^{\bullet}[r]],
[M(p^{l}+1,\overline{0})^{\bullet}[r]]$ and $[P^{\bullet}(i,j)[r]]$,
where $r\in \mathds{Z}, 1\leq l\leq m-1, -\infty\leq j\leq
i<+\infty$.

(ii) The relations of the generators of $dr(KZ_{n}/J^{2})$ in (i)
can be determined by those in Proposition \ref{prop10.3} (ii) and
respectively
 in the following cases for all $j,j',s,s'\in \mathds{Z}$ and $s'\leq s$: \\
(1)~ when $s'=1$, by the decomposition
$$P^{\bullet}(j'+s'-1,j')\otimes P^{\bullet}(j+s-1,j)=P^{\bullet}_{\overline{j+j'}}[j+j']\oplus P^{\bullet}_{\overline{j+j'+s}}[j+j'+s-1];$$
(2)~ when $s=s'>1$, by the decomposition:
$$P^{\bullet}(j'+s'-1,j')\otimes P^{\bullet}(j+s-1,j)=P^{\bullet}(j+j'+s'-1,j+j')\oplus P^{\bullet}(j+j'+2s-1,j+j'+s)[-1];$$
(3)~ when $s>s'>1$, by one of the three possible decompositions:
 $$P^{\bullet}(j'+s'-1,j')\otimes P^{\bullet}(j+s-1,j)= P^{\bullet}(j+j'+s'-1,j+j') \oplus
P^{\bullet}(j+j'+s+s'-1,j+j'+s)[-1];$$
 $$P^{\bullet}(j'+s'-1,j')\otimes P^{\bullet}(j+s-1,j)= P^{\bullet}(j+j'+s-1,j+j') \oplus
P^{\bullet}(j+j'+s+s'-1,j+j'+s')[-1];$$
$$P^{\bullet}(j'+s'-1,j')\otimes P^{\bullet}(j+s-1,j)= P^{\bullet}(j+j'+s+s'-1,j+j') \oplus
P^{\bullet}(j+j'+s-1,j+j'+s')[-1].$$

(iii) For the indeterminates $y[r],z[r],w_{l}[r],\nu(i,j)[r]$, let
$I''$ be the ideal of the polynomial ring
$$\mathds{Z}[y[r],z[r],w_{l}[r],\nu(i,j)[r]]_{r\in \mathds{Z}, 1\leq
l\leq m-1,-\infty\leq j\leq i<+\infty}$$ generated by the relations
which are given in Proposition \ref{prop10.3} (ii) and the (ii) in this theorem through replacing
$[M(1,\overline{1})^{\bullet}[r]],[M(2,\overline{0})^{\bullet}[r]],[M(p^{l}+1,\overline{0})^{\bullet}[r]],
[P^{\bullet}(i,j)[r]]$ by $$y[r],z[r],w_{l}[r],\nu(i,j)[r],$$
respectively for
 $r\in \mathds{Z}, 1\leq l\leq m-1,-\infty\leq j\leq i<+\infty$. Then there is the ring isomorphism   $$dr(KZ_{n}/J^{2})\cong \mathds{Z}[y[r],z[r],w_{l}[r],\nu(i,j)[r]]_{r\in \mathds{Z}, 1\leq l\leq m-1, -\infty\leq j\leq i<+\infty}/I''.$$}
\end{Thm}
\begin{proof}
The proofs of the parts (i) and (iii) are similar with that of the
corresponding parts of Proposition \ref{prop10.3}. One only needs to
note that all the indecomposable objects of $D^{b}(KZ_{n}/J^{2})$
are $\{P^{\bullet}(i,j)[l]\}_{-\infty \leq j\leq i<+\infty, l\in
\mathds{Z}}$ and similarly the ring homomorphism $\psi$ can be
defined about the indeterminates $y[r],z[r],w_{l}[r],\nu(i,j)[r]$ as
follows:
$$\psi:\mathds{Z}[y[r],z[r],w_{l}[r],\nu(i,j)[r]]_{r\in \mathds{Z}, 1\leq l\leq m-1, -\infty\leq j\leq i<+\infty}\rightarrow dr(KZ_{n}/J^{2})$$ by $\psi(y[r])=[M(1,\overline{1})^{\bullet}[r]]$, $\psi(z[r])=[M(2,\overline{0})^{\bullet}[r]]$, $\psi(w_{l}[r])=[M(p^{l}+1,\overline{0})^{\bullet}[r]]$ and $\psi(\nu(i,j)[r])=[P^{\bullet}(i,j)[r]]$.

The proof of the part (ii) is due to Fact \ref{fact1} (i) and
Proposition \ref{prop10.3} (ii) and the above discussion before this
theorem with (\ref{relations'=1}), (\ref{s=s'>1}) and the Cases
2,3,4 in (II).
\end{proof}

Furthermore, if Conjecture \ref{conj10.4} is confirmed positively, we can unify
the (1),(2),(3) of Theorem \ref{thm10.6} (ii) with the decomposition
(\ref{conj}) in Conjecture \ref{conj10.4} eventually.
 \vspace{8mm}

{\bf Acknowledgements:}\; This project is supported by the National Natural Science
Foundation of
China (No.11271318, No.11171296 and No.J1210038) and the Specialized Research Fund for the
Doctoral Program of Higher Education of China (No.20110101110010) and the Zhejiang
Provincial
Natural Science Foundation of China (No.LZ13A010001).

\end{document}